\documentclass[
    11pt,
    a4paper
]{article}

\usepackage[utf8]{inputenc}
\usepackage[english]{babel}
\usepackage{amsmath}
\usepackage{amssymb}
\usepackage{amsthm}
\usepackage{bbm}
\usepackage{mathtools}
\usepackage{enumitem}
\usepackage[numbers]{natbib}
\usepackage{graphicx}
\usepackage{subcaption}
\usepackage{xcolor}
\usepackage{hyperref}
\usepackage[normalem]{ulem} 

\usepackage[top=3cm,bottom=3cm,left=3cm,right=3cm]{geometry}

\theoremstyle{definition}
\newtheorem{definition}{Definition}[section]
\newtheorem{remark}[definition]{Remark}

\newtheorem{assumption}{Assumption}

\theoremstyle{plain}
\newtheorem{theorem}[definition]{Theorem}

\newtheorem{lemma}[definition]{Lemma}

\definecolor{forestgreen}{RGB}{34,139,34}

\newcommand{\NN}{\mathbb{N}}

\newcommand{\RR}{\mathbb{R}}

\newcommand{\ind}{\mathbbm{1}}

\newcommand{\PP}{\mathbf{P}}
\newcommand{\EE}{\mathbf{E}}

\title{Maxima of stationary systems of randomly time-changed L\'evy particles}
\author{Ioan Scheffel$^{a}$}
\date{$^{a}$ {\small Institute for Stochastics and Applications, University of Stuttgart, D-70563 Stuttgart, Germany}}

\begin{document}
\maketitle
\begin{abstract}
 In this work, we consider maxima of systems of randomly time-changed L\'evy particles.
We give a general construction to obtain infinite-dimensional classes $\{Z^{\alpha}\}$ of stationary
max-infinitely divisible (max-id) processes.
  These classes are indexed by admissible mass functions $\alpha$, which induce state-dependent time changes of the underlying L\'evy particles. This gives a generalization of the well-known (L\'evy--)Brown--Resnick process $Z^{1}$. In contrast to $\alpha\equiv 1$, the variability of non-constant mass functions $\alpha$ changes the
  dependence structure of the max-id process and goes beyond the max-stable setting while preserving stationarity. We then
  explore
  the extent of the so-called max-domain of attraction (MDA) of a given (L\'evy--)Brown--Resnick process $Z^1$, by
  studying convergence of rescaled maxima of independent copies of
  $Z^{\alpha}$ to $Z^1$.
  Thus, our work combines potential theory for Markov processes and extreme value theory to yield a novel, infinite-dimensional, and interpretable class $\{Z^{\alpha}\}$ of
  stationary processes in the MDA of a given (L\'evy--)Brown--Resnick process $Z^{1}$.
  So far, results on the extent of such domains have been scarce in the literature.
\end{abstract}



\section{Introduction}
\subsection{Background and literature review}
Given a Markov process $(X_t)_{t\ge 0}$ and a collection of points $\{U^{(i)}\}_{i\in \NN}$,
we consider the collection of particles obtained by independently starting in each point $U^{(i)}$ an i.i.d. copy $(X_t^{(i)})_{t\ge 0}$ of the Markov process $(X_t)_{t\ge 0}$, that is,
\begin{align*}
  \left\{
  \left(
  X_t^{(i)}(U^{(i)})
  \right)_{t\ge 0}
  \right\}_{i\in \NN}
  \,.
\end{align*}
We call this a Markovian particle system.
For a single particle $(X_t(U))_{t\ge 0}$, stationarity holds if and only if there exists an invariant probability measure $\mu$ of the process $(X_t)_{t\ge 0}$, by setting $U\sim \mu$.
Extreme value analysis naturally shifts the viewpoint from a single particle to
a system of infinitely many particles, where the starting points $\{U^{(i)}\}_{i\in\NN}$ form a
so-called Poisson point process (PPP).
For
a distribution function $F$ with density $F'$,
it is, for example, possible to choose the PPP $\{U^{(i)}\}_{i\in\NN}$ such that
$
  F(x)
   =
\PP
  \left[
  \max_{i\in\NN}U^{(i)}\le x
  \right]
$, for $x\in\RR$.
If the arising Markovian particle system is stationary,
then the maximum process
\begin{align}
  \label{eq:max_id}
  Z_t
    \ = \
    \max_{i\in\NN}
    X_t^{(i)}(U^{(i)})
    \,,
    \qquad t\ge 0
    \,,
 \end{align}
 is stationary as well,
 with marginal distribution function $F$. Such a process (regardless of stationarity) is called max-infinitely divisible (max-id), in analogy to the notion of infinite divisibility for sums rather than maxima.
 It follows from~\cite{brownPropertyPoissonProcesses1970} that
 the Markovian particle system (and the max-id process) is stationary if and only if $\mu(\mathrm{d}x)=F'(x)/F(x)\ind\{F(x)>0\}\mathrm{d}x$ is invariant for $(X_t)_{t\ge 0}$ in the sense that
\begin{align}
  \label{eq:inv}
  \int_{\RR}
  \PP[X_t(x)\in A]
  \,\mu(\mathrm{d}x)
  \ = \
  \mu(A)
  \qquad\text{for all Borel sets}\ A\ \ \text{and all} \ t\ge 0\,.
\end{align}
In this setting the invariant measure $\mu$ is necessarily infinite,
so that, unlike in the single-particle case discussed above,
it cannot be normalized to a probability measure.

A possible choice for $F$ is the Gumbel distribution function
 \begin{align*}
 F(x)
 \ = \
 \exp(-\exp(-x))
 \,,\qquad x\in\RR
 \,,
 \end{align*}
 with
 $\mu(\mathrm{d}x)\ = \ \exp(-x)\mathrm{d}x$.
 Any invariant process--measure pair~$((X_t)_{t\ge 0}, \exp(-x)\mathrm{d}x)$ in the sense of \eqref{eq:inv} then yields a stationary max-id process with Gumbel margins, as in \eqref{eq:max_id}.
Brown and Resnick \cite{brownExtremeValuesIndependent1977}
were the first to construct such a process, based on the observation that this specific choice of $\mu$ is invariant for the drifted Brownian motion $(B_t-t/2)_{t\ge 0}$.
The resulting process $\zeta_{\mathrm{BR}}$, called Brown--Resnick process, is stationary with Gumbel margins:
\begin{align}
  \label{eq:BR}
  \zeta_{\mathrm{BR}}(t)
  \ := \
  \max_{i\in\NN}
  \left(
  B^{(i)}_t-t/2\ +\ U^{(i)}
  \right)
  \,,
  \qquad
  t\ge 0
  \,.
\end{align}
Even though the single particles $B^{(i)}_t-t/2 + U^{(i)}$ are non-stationary (they converge to $-\infty$ for $t\to\infty$), the
particle system $\{(B_t^{(i)}-t/2 + U^{(i)})_{t\ge 0}\}_{i\in\NN}$, and the max-id process $\zeta_{\mathrm{BR}}$ are stationary.
Brown and Resnick also showed that $\zeta_{\mathrm{BR}}$ is
not only max-id, but even max-stable.
In the context of this work, with the convention that the margins of the max-stable process $\zeta$ are Gumbel distributed and the process is indexed over $[0,\infty)$, this means that, for all $n\in\NN$,
\begin{align*}
  \left(
  \max_{k=1,\ldots,n}
  \zeta^{(k)}(t)
  \ - \
  \log(n)
  \right)_{t\ge 0}
  \ = \
  \left(
  \zeta(t)
  \right)_{t\ge 0}
  \qquad\text{in finite-dimensional distributions,}\qquad
\end{align*}
where $\zeta^{(k)}$, for $k\in\{1,\ldots,n\}$, are i.i.d. copies of $\zeta$.
De Haan showed in~\cite{haanSpectralRepresentationMaxstable1984}
that all max-stable processes (up to marginal transformations) are max-id, similar in structure to \eqref{eq:BR}.

Max-stable processes obtained from systems of L\'evy particles
beyond the drifted Brownian motion
have been studied in \cite{stoevErgodicityMixingMaxstable2008a},
where they show stationarity and Gumbel margins of the process
\begin{align*}
  \zeta_{L}(t)
  \ := \
  \max_{i\in\NN}
  \left(
  L_t^{(i)}
  \ + \
  U^{(i)}
  \right)
  \,,
  \qquad
  t\ge 0
  \,,
\end{align*}
 if the L\'evy process $(L_t)_{t\ge 0}$ satisfies $\EE[\exp(L_1)]=1$.
We refer to such processes as L\'evy--Brown--Resnick processes.
It is then easy to see that the L\'evy process $(L_t)_{t\ge 0}$ has invariant measure $\exp(-x)\mathrm{d}x$, so the method of \cite{brownPropertyPoissonProcesses1970} applies;
see also
\cite{engelkeMaxstableProcessesStationary2015} for a
generalization of L\'evy--Brown--Resnick processes
to the case where $(L_t)_{t\ge 0}$ satisfies
$\EE[\exp(L_1)]\neq 1$.

We note that
Kabluchko et al.
showed in their seminal paper
\cite{kabluchkoStationaryMaxstableFields2009}
that the
stationarity
of
$\zeta_{\mathrm{BR}}$, among more general max-stable processes, can be derived using entirely different techniques
that apply, for example, to Gaussian fields; see also \cite{kabluchkoStationarySystemsGaussian2010} and \cite{molchanovStationarityMultivariateParticle2013a} for a complete characterization of uni- and multivariate stationary Gaussian particle systems, respectively.

\subsection{Motivation and contribution of this work}
All specific examples of stationary max-id processes in the literature,
constructed from non-stationary single particles,
are (up to marginal transformations) max-stable.
The goal of this work is to generalize this construction
by allowing for random time change of the underlying L\'evy particles.
Starting from a L\'evy--Brown--Resnick process $\zeta_L$,
we construct infinite-dimensional classes $\{Z^{\alpha}\}$ of stationary max-id processes.
These classes are indexed by admissible mass functions $\alpha$ that induce state-dependent time changes of the underlying L\'evy particles.
In contrast to $\alpha\equiv 1$, which yields the
L\'evy--Brown--Resnick process $\zeta_L=Z^{1}$,
non-constant mass functions $\alpha$ allow for variability in the dependence structure of the max-id process that goes beyond the max-stable setting.
We then use this variability in $\alpha$ to
explore the extent of the so-called maximum domain of attraction (MDA) of a given L\'evy--Brown--Resnick process $\zeta_L$.
To be more precise, we study, for all admissible $\alpha$, the convergence
\begin{align*}
  \left(
  \max_{k=1,\ldots,n}
  Z_t^{\alpha,(k)}
  \ - \
  \log(n)
  \right)_{t\ge 0}
  \ \to \
  \left(
  \zeta_L(t)
  \right)_{t\ge 0}
  \qquad\text{weakly as $n\to\infty$,}\qquad
\end{align*}
where $(Z_t^{\alpha,(k)})_{t\ge 0}$, for $k\in\{1,\ldots,n\}$, are i.i.d. copies of $(Z^{\alpha}_t)_{t\ge 0}$.
More generally, we combine potential theory for Markov processes with extreme value theory to yield a novel, infinite-dimensional, and interpretable class of stationary processes in the MDA of a given L\'evy--Brown--Resnick process $\zeta_L$.
So far, results on the extent of such domains have been scarce in the literature.
\subsection{Structure of this work}
 In Section~\ref{sec:stat_part_system} we recall the mechanism behind the method of~\cite{brownPropertyPoissonProcesses1970}.
 In Section~\ref{sec:perturbation} we perturb the process--measure pair of a L\'evy--Brown--Resnick process $\zeta_L=Z^1$
  by random time change induced by admissible mass functions $\alpha$.
  In Section~\ref{sec:max-id} we give a general construction of
  stationary max-id process $(Z^{\alpha}_t)_{t\ge 0}$.
  Finally, in Section~\ref{sec:mda}
  we show that the extremal behavior of $(Z^{\alpha}_t)_{t\ge 0}$ links back to $Z^1$ via $(Z^{\alpha}_t)_{t\ge 0}\in \mathrm{MDA}(Z^{1})$.
\subsection{Notation}
Let $I\subset[0,\infty)$ be an interval and define $\mathbb{D}(I)$ to be the set of
c\`adl\`ag functions on $I$ endowed with the Skorokhod topology, and let $\mathbb{C}(I)$ be the continuous functions endowed with the uniform topology.
The term Markov process is used exclusively to refer to real-valued, time-homogeneous Markov processes with paths in
$\mathbb{D}([0,\infty))$.
For a Markov process $(X_t)_{t\ge 0}$ we denote $X_t(x)$ to be the process at time $t\ge 0$, started in $x\in\RR$ at time $0$.
We denote $\mathcal{B}(\mathcal{X})$ to be the Borel $\sigma$-algebra of the topological space $\mathcal{X}$.
\section{Stationary particle systems}
\label{sec:stat_part_system}
To showcase the method of~\cite{brownPropertyPoissonProcesses1970}, we first introduce Poisson point processes (PPP) and explain how to construct particle systems by attaching to each point an independent copy of a Markov process. Finally we recap a theorem of \cite{brownPropertyPoissonProcesses1970} which states that the resulting particle system is stationary if and only if the intensity measure of the PPP is invariant for the  Markov process.

\begin{definition}[Poisson point process]
  Let $(\mathcal{X}_{}, \mathcal{C}, \mu)$ be a measure space satisfying $\{x\}\in \mathcal{C}$ for all $x\in \mathcal{X}$ and $\mu(\mathcal{X})=\infty$.
  A random collection $\{U^{(i)}\}_{i\in\NN}
  $ of points in $\mathcal{X}$ is called Poisson point process having intensity measure $\mu$ if for every $m\in\NN$ and corresponding disjoint sets $C_1,\ldots,C_m \in \mathcal{C}$ and non-negative (including $+\infty$) integers $r_1,\ldots,r_m $ it holds that
  \begin{align*}
    \PP
    \left[
    \sum_{i\in\NN}
    \ind\{U^{(i)}\in C_j\}
           \
    =
    \
    r_j
    \qquad\text{for all}\
    j\in \{1,\ldots,m\}
    \right]
    \ = \
    \prod_{j=1}^m
    p(\mu(C_j),r_j)
    \,,
  \end{align*}
  where
  \begin{align*}
    p(\lambda,k)
    \ = \
    \begin{cases}
      \frac{\lambda^k \exp(-\lambda)}{k!}
      &\qquad\text{if}\qquad \lambda <\infty\quad \text{and}\quad  k<\infty\,,\\
      1
      &\qquad\text{if}\qquad \lambda =\infty\quad \text{and}\quad  k=\infty
      \qquad\text{or}\qquad \lambda =0 \quad \text{and}\quad  k=0\,,
      \\
      0&\qquad\text{elsewhere}.
    \end{cases}
  \end{align*}
\end{definition}
\begin{remark}
A collection $\{U^{(i)}\}_{i\in\NN}$, in contrast to a set $\{U^{(i)}\mid i\in\NN \}$, can contain the same element more than once.
The existence of a PPP with prescribed intensity measure is guaranteed if, for example, the intensity measure is $\sigma$-finite.
\end{remark}
Let $(\RR,\mathcal{B}(\RR),\mu)$ be a $\sigma$-finite
measure space.
Given
a PPP $\{U^{(i)}\}_{i\in\NN}$ on the real line
with intensity measure $\mu$,
we can attach to each point $U^{(i)}$ an i.i.d.~copy $(X_t^{(i)})_{t\ge 0}$ of a  Markov process $(X_t)_{t\ge 0}$, independent of $\{U^{(i)}\}_{i\in\NN}$, to obtain the random collection
\begin{align}
  \label{eq:ps_full}
\left\{
\left(
X^{(i)}_t(U^{(i)})
  \right)_{t\ge 0}
  \right\}_{i\in\NN}
  \qquad\text{of sample paths in}\qquad
  \mathbb{D}([0,\infty))
\,.
\end{align}
We refer to this collection as the particle system associated with
the process--measure pair $((X_t)_{t\ge 0},\mu)$, and denote it by $\mathfrak{P}((X_t)_{t\ge 0}, \mu)$.
For a finite sequence of times $0\le t_1<\cdots<t_\ell<\infty,\ell\in\NN$, we define a finite-dimensional particle system by
\begin{align}
  \label{eq:ps_fd}
  \mathfrak{P}_{t_1,\ldots,t_\ell}((X_t)_{t\ge 0},\mu)
  \ := \
\left\{
\left(
  X^{(i)}_{t_1}(U^{(i)}),
  \ldots,
  X^{(i)}_{t_\ell}(U^{(i)})
\right)
  \right\}_{i\in\NN}
  \qquad\text{with points in}\qquad
  \RR^\ell
\,.
\end{align}
We say that the full particle system $\mathfrak{P}((X_{t})_{t\ge 0}, \mu)$ is stationary if, for all such finite collections of time points and all $h\ge 0$,
it holds that
\begin{align}
  \label{eq:stat}
  \mathfrak{P}_{t_1,\ldots,t_\ell}((X_t)_{t\ge 0},\mu)
  \qquad\text{has the same distribution as}\qquad
  \mathfrak{P}_{t_1+h,\ldots,t_\ell+h}((X_t)_{t\ge 0},\mu)
  \,.
\end{align}
It will be useful to have more general transformations of the particle system. To this end,
  let $(E,\mathcal{E})$ be a measurable space and $g$ be a (deterministic) measurable function from $(\mathbb{D}([0,\infty)), \mathcal{B}(\mathbb{D}([0,\infty)))$ to $(E,\mathcal{E})$.
  We define the $g\text{--transformed}$ particle system by
  \begin{align*}
    \mathfrak{P}_{g}
    \left(
    (X_t)_{t\ge 0},\mu
    \right)
    \ = \
    \left\{
    g
    \left(
    X_t^{(i)}(U^{(i)})
    \right)
    _{t\ge 0}
    \right\}_{i\in\NN}
    \qquad\text{with points in}\qquad
    E
    \,.
  \end{align*}
  The following theorem is a direct consequence of \cite[Theorem~1]{brownPropertyPoissonProcesses1970}, see also Section~3 therein.
\begin{theorem}
  \label{thm:brown}
  Let $(X_t)_{t\ge 0}$ be a  Markov process,
  $(\RR,\mathcal{B}(\RR),\mu)$ a $\sigma$-finite measure space with $\mu(\RR)=\infty$,
  and $g$ a measurable function from $(\mathbb{D}([0,\infty)), \mathcal{B}(\mathbb{D}([0,\infty)))$ to a measurable space $(E,\mathcal{E})$.
  Then the following holds:
  \begin{enumerate}
    \item
    The particle system $\mathfrak{P}_g((X_t)_{t\ge 0},\mu)$ defines a PPP on $(E,\mathcal{E})$ with intensity measure
    \begin{align*}
      \mu_g(A)
      \ = \
      \int_{\RR}
      \PP
      \left[
      g
      \left(
      (X_{t}(x))_{t\ge 0}
      \right)
      \in A
      \right]
      \,
      \mu(\mathrm{d}x)
      \qquad\text{for}\ A\in \mathcal{E}
      \,.
    \end{align*}
        \item
    The particle system $\mathfrak{P}((X_t)_{t\ge 0},\mu)$ is stationary if and only if
    the process--measure pair $((X_t)_{t\ge 0},\mu)$ is invariant.
  \end{enumerate}
\end{theorem}
\section{Perturbation by random time change}
\label{sec:perturbation}
Given the invariant process--measure pair
$((L_t)_{t\ge 0}, \exp(-x)\mathrm{d}x)$ of a L\'evy--Brown--Resnick process, we
construct
a new pair with different invariant intensity and dependence structure.
Our approach is tailored to invariant measures having infinite mass and is based on
random time change.
The main result of this section, which is also used subsequently in the paper, is Theorem~\ref{thm:pert_inv}.
Readers interested only in the main result may skip the remainder of the section and return to it later for the precise mathematical statements.

We first introduce the necessary definitions and assumptions. Then,
we
establish existence of time-changed processes and provide a useful representation of random clocks.
\begin{definition}[Random clock]
\label{def:random_clock}
  Given a L\'evy process $(L_t)_{t\ge 0}$ and a bounded function $\alpha:\RR \to (0,\infty)$, we define
\begin{align*}
  A_x(t)
  \ = \
  \int_0^t
  \alpha(L_r+x)
  \,\mathrm{d}r
  \qquad\text{and}\qquad
  T_x(t)
  \ = \
  \inf\{s\ge 0 \mid A_x(s)> t\}
  \,,\qquad
  t\ge 0 \,, x\in\RR\,.
\end{align*}
  We call $(T_x(t))_{t\ge 0}$ the random clock, $(A_x(t))_{t\ge 0}$ the inverse random clock, and $\alpha$ the mass function.
\end{definition}
Note that $A_x$, for all $x\in\RR$, is a strictly increasing, continuous random function, with inverse $T_x$, which we denote as $T_x=(A_x)^{-1}$.
\begin{assumption}[Assumption on mass function]
  \label{asu:alpha}
  The mass function $\alpha$
 is continuous
  and
  satisfies
\begin{align}
  \label{cond:alpha}
  0
  \ < \
  \underline{\alpha}
  \ := \
  \inf_{z\in\RR}
  \alpha(z)
  \ \le \
  \sup_{z\in\RR}
  \alpha(z)
  \ =: \
  \overline{\alpha}
  \ < \
  \infty
  \qquad\text{and}\qquad
  \lim_{z\to\infty}
  \alpha(z)
  \ = \ 1
  \,.
\end{align}
\end{assumption}

\begin{remark}[Assumption~\ref{asu:alpha}]
  By virtue of the first part we
  obtain existence and infinite lifetime of the time-changed process that we are going to consider. The second part of the assumption is relevant only in Section~\ref{sec:mda}, where convergence plays a key role.
\end{remark}

\begin{lemma}[Existence of time-changed process]
  \label{lem:existence_tcp}
  Let $(L_t)_{t\ge 0}$ be a L\'evy process and $\alpha$ a
  mass function satisfying Assumption~\ref{asu:alpha}. Then the time-changed process
  \begin{align*}
    X_t(x)
    \ = \
    L_{T_x(t)}
    \ + \
    x
    \,,\qquad t\ge 0, x\in\RR\,,
  \end{align*}
  is a Markov process.
\end{lemma}
\begin{proof}
  For the statement that the time-changed process is a Markov process
  with paths in $\mathbb{D}([0,\xi))$, where $\xi$ is the (possibly random) lifetime of the process,
  see, for example,
  \cite[(V.2.11)]{blumenthalMarkovProcessesPotential1968}.
  To finish the proof we show that the lifetime $\xi$ is infinite, which follows from
  $T_x(t)\to \infty$ for $t\to\infty$ and all $x\in\RR$.
  Since $T_x=A_x^{-1}$, this is equivalent to
  $A_x(t)\to \infty$ for $t\to\infty$ and all $x\in\RR$.
  But this follows immediately from the lower bound on $\alpha$ in Assumption~\ref{asu:alpha} by taking the limit $t\to \infty$ in
  \begin{align*}
    0
    \ < \
    t\cdot \underline{\alpha}
    \ \le \
    \int_0^t
    \alpha(L_r + x)
    \,\mathrm{d}r
    \ = \
    A_x(t)
    \,.
  \end{align*}
\end{proof}
\begin{remark}[Random time change by continuous additive functionals]
  The process $(X_t)_{t\ge 0}$ in Lemma~\ref{lem:existence_tcp} is an example of random time change by a
  continuous additive functional (CAF).
  For background on potential theory of
  CAFs,
  see
\cite[Chapter~IV.3]{blumenthalMarkovProcessesPotential1968}.
\end{remark}

  Let $\lambda^+$ be the Lebesgue measure on $\mathcal{B}([0,\infty))$, and denote the push-forward measure of $\lambda^+$ by $A_x$ as $\lambda^+\circ A_x^{-1}$.
  When precision of notation is not needed, we will use the equivalent notation $\mathrm{d}r$ instead of $\lambda^{+}(\mathrm{d}r)$.
\begin{lemma}[Integral representation of random clock]
  \label{lem:int_rep}
  Let $(L_t)_{t\ge 0}$ be a L\'evy process, let $\alpha$ be a mass function satisfying Assumption~\ref{asu:alpha},
  let $A_x$ and $T_x$ be the (inverse) random clock as in Definition~\ref{def:random_clock},
  and let $(X_t)_{t\ge 0}$ be the time-changed process in the sense of Lemma~\ref{lem:existence_tcp}.
  Then $\lambda^+$ is absolutely continuous with respect to $\lambda^+\circ A_x^{-1}$ with density $r\mapsto \alpha(X_r(x))$ and
  \begin{align*}
    T_x(t)
    \ = \
    \int_0^t
    \frac{1}{\alpha(X_r(x))}
    \,\mathrm{d}r
    \,.
  \end{align*}
\end{lemma}
\begin{proof}
  By the definition of the push-forward measure, the monotonicity of $A_x$, and $A_x^{-1}=T_x$, it holds for all $t\ge 0$
\begin{align*}
  \lambda^+([0,t])
  &
  \ = \
  t
  \ = \
  A_x(T_x(t))
  \ = \
  \int_{[0,T_x(t)]}
  \alpha(L_r + x)
  \,
  \lambda^+(\mathrm{d}r)
  \\&
  \ = \
  \int_{[0,t]}
  \alpha(L_{T_x(r)}+x)
  \,
  (\lambda^+ \circ A_x^{-1})(\mathrm{d}r)
  \\&
  \ = \
  \int_{[0,t]}
  \alpha(X_r(x))
  \,
  (\lambda^+ \circ A_x^{-1})(\mathrm{d}r)
  \,.
\end{align*}
  Therefore $\lambda^+$ is absolutely continuous with respect to $\lambda^+ \circ A_x^{-1}$ with density $r\mapsto \alpha(X_r(x))$.
It follows that
\begin{align*}
T_x(t)
  &
  \ = \
  \lambda^+[[0,T_{x}(t)]]
  \ = \
  \left(
  \lambda^+ \circ A_x^{-1}
  \right)[[0,t]]
  \ = \
  \int_{[0,t]}
  \frac{1}{\alpha(X_r(x))}
  \,\mathrm{d}r
  \,.
\end{align*}

\end{proof}
\begin{assumption}[Assumption on L\'evy process]
  \label{asu:levy}
  The process $(L_t)_{t\ge 0}$ is a L\'evy process satisfying $\EE[\exp(L_1)]=1$.
\end{assumption}
\begin{remark}[Assumption~\ref{asu:levy}]
  This assumption leads to a L\'evy--Brown--Resnick process,
  as defined in the introduction.
  In the sequel we will frequently use that, for L\'evy processes, $\EE[\exp(L_1)]=1$ implies
  $\EE[\exp(L_t)]=1$ for all $t\ge 0$. It is easy to see that the process--measure pair $((L_t)_{t\ge 0}, \exp(-x)\mathrm{d}x)$ is invariant. Indeed,
  \begin{align*}
    &
    \int_{\RR}
    \PP[L_t + x \in B]
    \exp(-x)
    \,\mathrm{d}x
    \\&
    \ = \
    \EE
    \left[
    \int_{\RR}
    \ind\{L_t + x\in B\}
    \exp(-x)
    \,\mathrm{d}x
    \right]
    \\&
    \ = \
    \EE
    \left[
    \int_{\RR}
    \ind\{ x\in B\}
    \exp(-(x-L_t))
    \,\mathrm{d}x
    \right]
    \ = \
    \EE
    \left[
    \exp(L_t)
    \right]
    \cdot
    \int_B
    \exp(-x)
    \,\mathrm{d}x
    \ = \
    \int_B
    \exp(-x)
    \,\mathrm{d}x
    \,.
  \end{align*}
\end{remark}
\begin{theorem}[Perturbed invariant process--measure pair]
  \label{thm:pert_inv}
  Let $\alpha$ be a mass function satisfying Assumption~\ref{asu:alpha},
  let $(L_t)_{t\ge 0}$ be a L\'evy process satisfying Assumption~\ref{asu:levy},
  and let $(X_t)_{t\ge 0}$ be the time-changed process in the sense of Lemma~\ref{lem:existence_tcp}. Then the measure $\alpha(x)\exp(-x)\mathrm{d}x$ is invariant for $(X_t)_{t\ge 0}$.
\end{theorem}
\begin{proof}
  Since $(X_{t})_{t\ge 0}$ is a Markov process by Lemma~\ref{lem:existence_tcp}, it follows from~\cite[Proposition~2.1]{azemaMesureInvarianteProcessus1969} that
  the invariance of the measure $\alpha(x)\exp(-x)\,\mathrm{d}x$ for $(X_{t})_{t\ge 0}$ is equivalent to
\begin{align*}
  \int_{\RR}
  \left(
  \int_0^{\infty}
  \exp(-t)
  \cdot
  \PP[X_{t}(x)\in B]
  \,\mathrm{d}t
  \right)
  \alpha(x)
  \exp(-x)
  \,\mathrm{d}x
  \ = \
  \int_{B}
  \alpha(x)
  \exp(-x)
  \,\mathrm{d}x
  \,,
\end{align*}
  for all $B\in \mathcal{B}(\RR)$.
  To show this, note that $X_t(x)=L_{T_x(t)}+x$, and that
\begin{align*}
  &
  \int_0^{\infty}
  \exp(-t)
  \ind\{L_{T_x(t)}+x\in B\}
  \,
  \lambda^+(\mathrm{d}t)
  \\&
  \ = \
  \int_0^{\infty}
  \exp(-t)
  \ind\{L_{T_x(t)}+x\in B\}
  \alpha(L_{T_x(t)}+x)
  \,
  (\lambda^+ \circ A_{x}^{-1})(\mathrm{d}t)
  \\&
  \ = \
  \int_0^{\infty}
  \exp(-A_x(t))
  \ind\{L_{t}+x\in B\}
  \alpha(L_{t}+x)
  \,
  \lambda^+(\mathrm{d}t)
  \,,
\end{align*}
  where for the first equality we use that $\lambda^+$ is absolutely continuous with respect to $\lambda^+\circ A^{-1}_x$, with density $r\mapsto \alpha(X_r(x))=\alpha(L_{T_x(r)}+x)$ (see
  Lemma~\ref{lem:int_rep}),
  and for the second equality we use
  the properties of push-forward measures and $A_x(0)=0$ and $\lim_{t\to\infty}A_x(t)=\infty$.
Therefore, by Fubini-Tonelli and change of variables, it holds
\begin{align*}
  &
  \int_{\RR}
  \exp(-x)
  \alpha(x)
  \left(
  \int_0^{\infty}
  \exp(-t)
  \cdot
  \PP[X_t(x)\in B]
  \,\mathrm{d}t
  \right)
  \,\mathrm{d}x
  \\&
  \ = \
  \EE
  \left[
  \int_{\RR}
  \exp(-x)
  \alpha(x)
  \left(
  \int_0^{\infty}
  \exp(-t)
  \ind\{L_{T_x(t)}+x\in B\}
  \,\mathrm{d}t
  \right)
  \,\mathrm{d}x
  \right]
  \\&
  \ = \
  \EE
  \left[
  \int_{\RR}
  \exp(-x)
  \alpha(x)
  \left(
  \int_0^{\infty}
  \exp(-A_{x}(t))
  \ind\{L_{t}+x\in B\}
  \alpha(L_t+x)
  \,\mathrm{d}t
  \right)
  \,\mathrm{d}x
  \right]
  \\&
  \ = \
  \EE
  \left[
  \int_0^{\infty}
  \left(
  \int_{\RR}
  \exp(-x)
  \alpha(x)
  \exp
  \left(
  -\int_0^t
  \alpha(L_{t-s}+x)
  \,\mathrm{d}s
  \right)
  \ind\{L_{t}+x\in B\}
  \alpha(L_t+x)
  \,\mathrm{d}x
  \right)
  \,\mathrm{d}t
  \right]
  \\&
  \ = \
  \EE
  \left[
  \int_0^{\infty}
  \left(
  \int_{\RR}
  \exp(-(x-L_t))
  \alpha(x-L_t)
  \exp
  \left(
  -\int_0^t
  \alpha(x-(L_t - L_{t-s}))
  \,\mathrm{d}s
  \right)
  \ind\{x\in B\}
  \alpha(x)
  \,\mathrm{d}x
  \right)
  \,\mathrm{d}t
  \right]
  \\&
  \ = \
  \int_{B}
  \alpha(x)
  \exp(-x)
  \left(
  \int_0^{\infty}
  \EE
  \left[
  \exp(L_t)
  \alpha(x-L_t)
  \exp
  \left(
  -\int_0^t
  \alpha(x-(L_t - L_{t-s}))
  \,\mathrm{d}s
  \right)
  \right]
  \,\mathrm{d}t
  \right)
  \,\mathrm{d}x
  \,,
\end{align*}
where for the third equality we use that
$
\int_0^t\alpha(L_{r}+x)\,\mathrm{d}r
\ = \
\int_0^t\alpha(L_{t-s}+x)\,\mathrm{d}s
$, and for the fourth equality we make the change of variables $x\mapsto L_t+x$.
For $t\ge 0$,
define the process $(Y^t_s)_{s\in[0,t]}$ by $Y^t_s=L_t-L_{t-s}$.
Then $(Y^t_s)_{s\in[0,t]}$ has the same finite-dimensional distributions as $(L_s)_{s\in[0,t]}$.
Note that $L_t=L_t-L_{t-t}=Y^t_t$, so that
\begin{align*}
  &
  \int_0^{\infty}
  \EE
  \left[
  \exp(L_t)
  \alpha(x-L_t)
  \exp
  \left(
  -\int_0^t
  \alpha(x-(L_t - L_{t-s}))
  \,\mathrm{d}s
  \right)
  \right]
  \,\mathrm{d}t
  \\&
  \ = \
  \int_0^{\infty}
  \EE
  \left[
  \exp(Y^t_t)
  \alpha(x-Y^t_t)
  \exp
  \left(
  -\int_0^t
  \alpha(x-(Y^t_s))
  \,\mathrm{d}s
  \right)
  \right]
  \,\mathrm{d}t
  \\&
  \ = \
  \int_0^{\infty}
  \EE
  \left[
  \exp(L_t)
  \alpha(x-L_t)
  \exp
  \left(
  -\int_0^t
  \alpha(x-L_s)
  \,\mathrm{d}s
  \right)
  \right]
  \,\mathrm{d}t
  \,.
 \end{align*}
 Now, let $\widehat{A}_x$ and $\widehat{T}_x$ as in Definition~\ref{def:random_clock}, with the same mass function $\alpha$ as before, but with L\'evy process $(-L_t)_{t\ge 0}$, that is,
 \begin{align*}
   \widehat{A}_{x}(t)
   \ = \
   \int_0^{t}
   \alpha(x - L_r)
   \,\mathrm{d}r
   \qquad\text{and}\qquad
   \widehat{T}_x(t)
   \ = \
   \left(
   \widehat{A}_{x}
   \right)^{-1}
   (t)
   \,.
 \end{align*}
 It is easy to see that $\widehat{A}_x(0)=0$ and $\lim_{t\to\infty}\widehat{A}_x(t)=\infty$, so that by the properties of push-forward measures and Lemma~\ref{lem:int_rep} it follows that
 \begin{align*}
   &
  \int_0^{\infty}
  \exp(L_t)
  \alpha(x-L_t)
  \exp
  \left(
  -
   \widehat{A}_x(t)
  \right)
   \,
   \lambda^+(\mathrm{d}t)
   \\&
   \ = \
  \int_0^{\infty}
   \exp
   \left(
   L_{\widehat{T}_x(t)}
   \right)
   \alpha
   \left(
   x-L_{\widehat{T}_x(t)}
   \right)
  \exp
  \left(
  -t
  \right)
  \,
   \left(
   \lambda^+ \circ
   \left(
   \widehat{A}_x
   \right)^{-1}
   \right)
   (\mathrm{d}t)
   \\&
   \ = \
  \int_0^{\infty}
   \exp
   \left(
   L_{\widehat{T}_x(t)}
   \right)
  \exp
  \left(
  -
  t
  \right)
  \,
   \lambda^+(\mathrm{d}t)
  \,.
 \end{align*}
 Next, we apply optional stopping to $\exp(L_{\widehat{T}_x(t)})$, whence it follows
 \begin{align*}
   \EE[\exp(L_{\widehat{T}_x(t)})]=\EE[\exp(L_0)]=1
   \,,
 \end{align*}
 and therefore
 \begin{align*}
  \int_0^{\infty}
  \EE
  \left[
   \exp(L_{\widehat{T}_x(t)})
  \right]
  \exp
  \left(
  -
  t
  \right)
  \,\mathrm{d}t
   \ = \
   \int_0^{\infty}
   \exp(-t)
   \,\mathrm{d}t
   \ = \ 1
   \,,
 \end{align*}
 which finishes the proof.
 To this end, we have to show that $\widehat{T}_x(t)$ is a bounded stopping time and that $(\exp(L_t))_{t\ge 0}$ is a martingale.
 Note that $\widehat{T}_{x}(t)$, by Lemma~\ref{lem:int_rep}, is a stopping time with
 \begin{align*}
   \widehat{T}_{x}(t)
   \ = \
   \int_0^t
   \frac{1}{\alpha(x - L_{\widehat{T}_{x}(r)})}
   \,\mathrm{d}r
   \ \le \
   t
   \cdot \frac{1}{\underline{\alpha}}
   \,,
 \end{align*}
 and
 that for $0\le s\le t$ it holds that
 \begin{align*}
   \EE[\exp(L_t)\mid L_s]
   \ = \
   \EE[\exp(L_t-L_s)\exp(L_s)\mid L_s]
   \ = \
   \EE[\exp(L_{t-s})] \exp(L_s)
   \ = \
   \exp(L_s)
   \,,
 \end{align*}
 where the second equality follows from the L\'evy property of $(L_t)_{t\ge 0}$ and the properties of conditional expectation, and the last equality follows from $\EE[\exp(L_{t-s})]=1$ for all $0\le s \le t$.
 Therefore $(\exp(L_t))_{t\ge 0}$ is a martingale, and we can apply optional stopping to conclude the proof.
\end{proof}
\begin{remark}
  Another possibility of perturbing an invariant process--measure pair $((L_t)_{t\ge 0}, \exp(-x)\mathrm{d}x)$ of a L\'evy--Brown--Resnick process is via the theory of
  Feller processes;
  see \cite[Chapter~4]{bottcherLevyMattersIII2013}
for an overview of transformations of Feller processes.
  It is beyond the scope of this work to explore this direction.
\end{remark}
\section{Stationary max-id processes}
\label{sec:max-id}
In the special case where the intensity measure $\mu$ of an invariant process--measure pair $((X_t)_{t\ge 0},\mu)$ has finite mass on sets bounded away from $-\infty$, the maximum over the corresponding particle system
$\mathfrak{P}_t((X_t)_{t\ge 0},\mu)$
is almost surely finite at all times $t \ge 0$. We call processes obtained in this way max-infinitely divisible (max-id), in the sense that they are componentwise maxima over PPP with points in $\mathbb{D}([0,\infty))$.
More precisely, let
$\alpha$ be a mass function satisfying Assumption~\ref{asu:alpha}, let
$(L_t)_{t\ge 0}$ be a L\'evy process
satisfying Assumption~\ref{asu:levy},
let $(X_t)_{t\ge 0}$ be the time-changed process in the sense of Lemma~\ref{lem:existence_tcp}, and consider
\begin{align*}
  Z_t
  \ = \
  \max
  \mathfrak{P}_{t}((X_{t})_{t\ge 0}, \alpha(x)\exp(-x)\mathrm{d}x)
  \ = \
  \max_{i\in\NN}
  X_t^{(i)}(U^{(i)})
  \,,
\end{align*}
where $\{U^{(i)}\}_{i\in\NN}$ is a PPP on the real line with intensity measure $\alpha(x)\exp(-x)\,\mathrm{d}x$, and $(X_t^{(i)})_{t\ge 0}$ are i.i.d.~copies of $(X_t)_{t\ge 0}$.
\begin{remark}[PPP from jump times of Poisson process]
  \label{rem:order}
  Such a PPP $\{U^{(i)}\}_{i\in\NN}$ can be obtained by a monotone transformation of the jump times of a unit Poisson process on the real line.
  To be more precise,
  \begin{align*}
    U^{(i)}
    \ = \
    F^{\leftarrow}
    (\exp(-\Gamma_i))
    \,,
    \qquad
    (\Gamma_i)_{i\in\NN}
    \ \text{jump times of unit Poisson process}
    \,,
  \end{align*}
  where $F^{\leftarrow}$ is the generalized inverse of $F$, that is $F^{\leftarrow}(q)=\inf\{x\in\RR\mid F(x)\ge q\}$.
  Therefore, the natural ordering of the points is preserved, where in our case, we can write without loss of generality $U^{(1)}=\max\{U^{(i)}\}_{i\in\NN}$, so that
  \begin{align*}
  \PP[U^{(1)}\le x]
    &
    \ = \
    \PP[F^{\leftarrow}(\exp(-\Gamma_1))\le x]
    \ = \
    \PP[
    \Gamma_1
    \ge -\log(F(x))
    ]
    \ = \
    \exp(-(-\log(F(x))))
    \\&
    \ = \
    F(x)\,.
  \end{align*}
\end{remark}
After a technical lemma on exceedance probabilities for random time-changed processes,
the main result of this section
(Theorem~\ref{thm:cad})
shows that the process $(Z_t)_{t\ge 0}$ has paths in $\mathbb{D}([0,\infty))$ and is stationary.
\begin{lemma}[Bound on exceedance probability]
  \label{lem:dynkin}
  Let $\alpha$ be a mass function satisfying Assumption~\ref{asu:alpha}, let $(L_t)_{t\ge 0}$ be a L\'evy process satisfying Assumption~\ref{asu:levy}, and let $(X_t)_{t\ge 0}$
  be the time-changed process in the sense of Lemma~\ref{lem:existence_tcp}.
  Then, for all $t\ge 0$, there exists $v>0$ and $C_{v,t}>0$ such that
  for all $u,x\in\RR$ satisfying $u-x > v$ it holds that
  \begin{align*}
    \PP
    \left[
    \sup_{s\in[0,t]}
    X_s(x)
    \ \ge \ u
    \right]
    \ \le \
    C_{v,t}
    \cdot
    \PP
    \left[
    L_{t\cdot \overline{\alpha}}
    \ge
    u - x - v
    \right]
    \,.
  \end{align*}
  \end{lemma}
\begin{proof}
  Denote
\begin{align*}
  \tau(u,x)
  \ := \
  \inf
  \{
  t\ge 0 \mid L_{t} + x \ge u
  \}
  \qquad\text{and}\qquad
  \widetilde{\tau}(u,x)
  \ := \
  \inf
  \{
  t\ge 0 \mid L_{T_x(t)} + x \ge u
  \}
  \,,
\end{align*}
  that is, the first hitting times of the level $u\in\RR$ of $(L_t)_{t\ge 0}$ and $(X_t)_{t\ge 0}$ respectively.
  Clearly it holds that
\begin{align*}
  T_x(\tau(u,x))
  \ = \
  \widetilde{\tau}(u,x)
  \,.
\end{align*}
  Since $A_x(t)=(T_x)^{-1}(t)$ is monotone it holds that
\begin{align*}
 \widetilde{\tau}(u,x)\ \le\  t
 \qquad\text{is equivalent to}\qquad
  \tau(u,x)\ \le\  A_x(t)
  \,.
\end{align*}
  Then by Assumption~\ref{asu:alpha}
\begin{align}
  \label{eq:C0}
  A_x(t)
  \ = \
  \int_0^t \alpha(L_r+x)\,\mathrm{d}r
  \ \le \
  t
  \cdot
  \overline{\alpha}
  \,.
\end{align}
Now choose $v>0$ large enough, such that
  \begin{align}
    \label{eq:C1}
    0
    \ < \
\PP
    \left[
    \inf_{s\in[0,t\cdot\overline{\alpha}]}
    L_s
    \ \ge \ -v
    \right]
    \,.
  \end{align}
  It follows from \eqref{eq:C0} that
  \begin{align}
    \label{eq:C2}
    \begin{split}
    &
    \PP
    \left[
    \sup_{s\in[0,t]}
    X_s(x)
    \ \ge \ u
    \right]
    \\&
    \ = \
    \PP
    \left[
    \sup_{s\in[0,t]}
    L_{T_x(s)}+x
    \ \ge \ u
    \right]
    \ = \
    \PP
    [
    \widetilde{\tau}(u,x)
    \ \le \
    t
    ]
    \ = \
    \PP
    [
    \tau(u,x)
    \ \le \
    A_x(t)
    ]
    \\&
    \ \le \
    \PP
    [
    \tau(u,x)
    \ \le \
    t\cdot \overline{\alpha}
    ]
    \\&
    \ = \
\PP
    \left[
    \sup_{s\in[0,t\cdot \overline{\alpha}]}
    L_{s}+x
    \ \ge \ u
    \right]
       \,.
    \end{split}
  \end{align}
  By \cite[(2.1)]{willekensSupremumInfinitelyDivisible1987} it holds
  for $u,x\in \RR$ satisfying $u-x>v$ that
  \begin{align}
    \label{eq:C3}
    &
\PP
    \left[
    \sup_{s\in[0,t\cdot \overline{\alpha}]}
    L_{s}+x
    \ \ge \ u
    \right]
    \ \le \
    \frac{
\PP
    \left[
    L_{t\cdot \overline{\alpha}}
    \ \ge \ u - x - v
    \right]
    }{
\PP
    \left[
    \inf_{s\in[0,t\cdot \overline{\alpha}]}
    L_s
    \ \ge \ -v
    \right]
    }
    \ = \
    C_{v,t}
    \cdot
\PP
    \left[
    L_{t\cdot \overline{\alpha}}
    \ \ge \ u - x - v
    \right]
     \,,
  \end{align}
  where, by \eqref{eq:C1},
  \begin{align*}
    C_{v,t}
    \ := \
    \frac{1}{
\PP
    \left[
    \inf_{s\in[0,t\cdot \overline{\alpha}]}
    L_s
    \ \ge \ -v
    \right]
    }
    \ < \
    \infty
    \,.
  \end{align*}
  Putting together \eqref{eq:C2} and \eqref{eq:C3} finishes the proof.
\end{proof}
\begin{theorem}[Existence and stationarity of max-id processes]
  \label{thm:cad}
  Let $\alpha$ be a mass function satisfying Assumption~\ref{asu:alpha}, let $(L_t)_{t\ge 0}$ be a L\'evy process satisfying Assumption~\ref{asu:levy}, and let $(X_t)_{t\ge 0}$
  be the time-changed process in the sense of Lemma~\ref{lem:existence_tcp},
  and consider the process $(Z_t)_{t\ge 0}$ defined by
\begin{align*}
  Z_t
  \ = \
  \max
  \mathfrak{P}_t((X_t)_{t\ge 0}, \alpha(x)\exp(-x)\mathrm{d}x)
  \,.
\end{align*}
  Then the process $(Z_t)_{t\ge 0}$ has paths in $\mathbb{D}([0,\infty))$ and is stationary.
If, furthermore, $(L_t)_{t\ge 0}$ has continuous paths, then $(Z_t)_{t\ge 0}$ has paths in $\mathbb{C}([0,\infty))$.
\end{theorem}
\begin{proof}
  We adapt the proof of \cite[Proposition~1.14]{engelkeMaxstableProcessesStationary2015} to our setting. To this end,
  it is sufficient to prove that $(Z_s)_{s\in[0,t]}$ has paths in $\mathbb{D}([0,t])$ for all $t>0$. Fix $t>0$ and $u\in\RR$, and consider
\begin{align*}
  I_u
  \ := \
  \left\{
  i \in \NN
  \ \bigg| \
  \sup_{s\in [0,t]}
  X^{(i)}_{s}(U^{(i)})
  \ \ge \ u
  \right\}
  \,.
\end{align*}
  This set contains the indices of particles that somewhere in the compact interval $[0,t]$ have values above the threshold $u$ and
  hence are the only particles that, in case that $\inf_{s\in[0,t]} Z_s\ge u$, could contribute to the maximum.
  Applying Theorem~\ref{thm:brown} we get that the cardinality of the set $I_u$ has Poisson distribution, which, by an application of Lemma~\ref{lem:dynkin}, has finite rate:
  \begin{align*}
  \lambda_u
    &
    \ := \
  \int_{\RR}
  \PP
  \left[
  \sup_{s\in[0,t]}
  X_s(x)
  \ \ge\  u
  \right]
    \alpha(x)\exp(-x)\,\mathrm{d}x
    \\&
    \ \le \
    \overline{\alpha}
    \cdot
    \left(
    \int_{-\infty}^{u-v}
  \PP
  \left[
  \sup_{s\in[0,t]}
  X_s(x)
  \ \ge\  u
  \right]
    \exp(-x)
  \,\mathrm{d}x
  \ + \
    \exp(-(u-v))
    \right)
    \,,
    \end{align*}
    where we chose $v>0$ as in Lemma~\ref{lem:dynkin}.
    Applying Lemma~\ref{lem:dynkin} gives
    \begin{align*}
      &
    \int_{-\infty}^{u-v}
  \PP
  \left[
  \sup_{s\in[0,t]}
  X_s(x)
  \ \ge\  u
  \right]
      \exp(-x)
  \,\mathrm{d}x
      \\&
  \ \le \
  C_{v,t}
  \cdot
      \int_{-\infty}^{u-v}
      \PP[L_{t\cdot \overline{\alpha}} \ge u - x - v]
      \exp(-x)
      \,\mathrm{d}x
      \\&
      \ = \
  C_{v,t}
  \cdot
      \exp(v-u)
  \cdot
      \int_{-\infty}^{u-v}
      \PP[\exp(L_{t\cdot \overline{\alpha}}) \ge \exp(u - x - v)]
      \exp(u-x-v)
      \,\mathrm{d}x
      \\&
      \ = \
  C_{v,t}
  \cdot
      \exp(v-u)
  \cdot
      \int_{1}^{\infty}
      \PP[\exp(L_{t\cdot \overline{\alpha}}) \ge y]
      \,\mathrm{d}y
      \\&
      \ \le \
  C_{v,t}
  \cdot
      \exp(v-u)
      \cdot
      \EE[\exp(L_{t\cdot \overline{\alpha}})]
      \ = \
  C_{v,t}
  \cdot
      \exp(v-u)
      \,.
    \end{align*}
 It follows $\lambda_u<\infty$, and therefore $|I_u|<\infty$ almost surely.
 Denote $J_u:=\{\inf_{s\in[0,t]}Z_s \ge u\}$, and
 note that $\PP[\cup_{u\in\mathbb{Z}}J_u]=1$. If $J_u$ for some $u\in \mathbb{Z}$, it therefore holds
 \begin{align*}
   Z_s
   \ = \
   \max_{i\in I_u} X^{(i)}_{s}(U^{(i)})
   \qquad\text{for all}\
   s\in [0,t]
   \,.
 \end{align*}
  Since $(X_s)_{s\in[0,t]}$ has paths in $\mathbb{D}([0,t])$, and $\mathbb{D}([0,t])$ is invariant under taking the maximum over finitely many functions, we get that $(Z_s)_{s\in[0,t]}$ has paths in $\mathbb{D}([0,t])$.
  If $(L_t)_{t\ge 0}\in \mathbb{C}([0,\infty))$, then $(X_t)_{t\ge 0}\in \mathbb{C}([0,\infty))$ and therefore $(Z_t)_{t\ge 0}\in \mathbb{C}([0,\infty))$ for the same reason.
  The stationarity of $(Z_t)_{t\ge 0}$ follows then immediately from $(Z_t)_{t\ge 0}$ being a well-defined process, the invariance of the process--measure pair $((X_t)_{t\ge 0}, \alpha(x)\exp(-x)\,\mathrm{d}x)$ by Theorem~\ref{thm:pert_inv}, and Theorem~\ref{thm:brown}(ii).
  This finishes the proof.
\end{proof}
\section{Max-domain of attraction}
\label{sec:mda}
In this final section we
study the extremal properties of the max-id processes constructed in Section~\ref{sec:max-id} through max-domains of attraction.
Let $\alpha$ be a mass function satisfying Assumption~\ref{asu:alpha},
let $(L_{t})_{t\ge 0}$ be a L\'evy process satisfying Assumption~\ref{asu:levy}, let $(X_t)_{t\ge 0}$ be the time-changed process in the sense of Lemma~\ref{lem:existence_tcp}, and let $(Z_t)_{t\ge 0}$ be the max-id process in the sense of Section~\ref{sec:max-id}.
Define, for $n\in\NN$,
the rescaled process $(\zeta_n(t))_{t\ge 0}$, by
\begin{align*}
    \zeta_n(t)
    \ := \
    \max_{k=1,\ldots,n} Z_t^{(k)} - \log(n)
    \,,
  \end{align*}
  where $((Z_t^{(k)})_{t\ge 0}),{k=1,\ldots,n}$, are i.i.d.~copies of $(Z_t)_{t\ge 0}$, and define
  the process
$(\zeta_{L}(t))_{t\ge 0}$
by
 \begin{align*}
    \zeta_L(t)
    \ := \
    \max \mathfrak{P}_{t}((L_t)_{t\ge 0}, \exp(-x)\mathrm{dx})
     \,,
  \end{align*}
  where $(L_t)_{t\ge 0}$ is the L\'evy process that was time-changed to obtain $(X_t)_{t\ge 0}$.
We show that $(Z_t)_{t\ge 0}$ is in the max-domain of attraction of the max-stable process $(\zeta_{L}(t))_{t\ge 0}$
in the sense that $(\zeta_n(t))_{t\ge 0}$
  converges weakly to $(\zeta_L(t))_{t\ge 0}$.
  \begin{remark}
  The weakest notion of convergence is convergence in finite-dimensional distributions (f.d.d), while in certain cases this can be strengthened to convergence in the space of c\`adl\`ag functions $\mathbb{D}([0,\infty))$, endowed with the Skorokhod topology.
\end{remark}
The main result of this section is Theorem~\ref{thm:max_domain}, where
we show
that $(\zeta_n(t))_{t\ge 0}$
  converges weakly to $(\zeta_L(t))_{t\ge 0}$, that is, $(Z_t)_{t\ge 0}$ is in the MDA of $\zeta_{L}$.
  The following auxiliary result isolates required convergence statements that we will use in the proof of Theorem~\ref{thm:max_domain}.
\begin{lemma}
  \label{lem:conv}
 Let $\alpha$ be a mass function satisfying Assumption~\ref{asu:alpha},
  let
$(L_t)_{t\ge 0}$ be a L\'evy process
satisfying Assumption~\ref{asu:levy},
    let $(X_t)_{t\ge 0}$ be the time-changed process in the sense of Lemma~\ref{lem:existence_tcp}, and write
    \begin{align*}
      \eta(x)
      \ := \
      \alpha(x)
      \exp(-x)
      \qquad\text{and}\qquad
  \eta_n(x)
  \ := \
  n\cdot \eta(x+\log(n))
  \,.
    \end{align*}
    As $n\to\infty$, we have
    \begin{enumerate}
      \item
\begin{align}
  \begin{split}
  \label{cond:thm:mda:1}
  \eta_n(x)
  &
  \ \to \
  \exp(-x)
  \qquad\forall x\in \RR\,,
  \\
  \int_{z}^\infty
  \eta_n(x)
  \,\mathrm{d}x
  &
  \ \to \
  \exp(-z)
  \qquad\forall z\in \RR\,,
  \end{split}
\end{align}
\item
\begin{align}
  \label{cond:thm:mda:2}
   X^n_t(x)
   \
   :=
   \
   X_t(x+\log(n))-\log(n)
   \ \to \
   L_t
   \qquad\text{weakly in}\qquad
   \mathbb{D}([0,\infty))
   \,.
\end{align}
    \end{enumerate}
  \end{lemma}
\begin{proof}
  The first part follows from Assumption~\ref{asu:alpha} and dominated convergence.
    To show the
    second part,
    we apply
  \cite[Theorem~1.5]{ethierMarkovProcessesCharacterization1986}.
  To this end, note that
  \begin{align*}
    X^{n}_t(x)
    &
    \ = \
    X_t(x+\log(n))
    \ - \
    \log(n)
    \ = \
    L_{T_{x+\log(n)}(t)}
    +x
       \,,
  \end{align*}
  where
  \begin{align*}
 \\&
    T_{x+\log(n)}(t)
    \ = \
    \int_0^t
    \frac{1}{\alpha
    \big(
    X^{}_s(x+\log(n))\big)
    }
    \,\mathrm{d}s
    \ = \
    \int_0^t
    \frac{1}{\alpha
    \left(
    X^{n}_s(x)+\log(n)
    \right)
    }
    \,\mathrm{d}s
    \,,
  \end{align*}
  and by Assumption~\ref{asu:alpha}
  \begin{align*}
    \lim_{n\to\infty}
    \sup_{z\in K}
    \left|
    \frac{1}{\alpha(z+\log(n))}
    \ - \ 1
    \right|
    \ = \ 0
    \,,
  \end{align*}
  for all compact sets $K\subset {\RR}$.
  Since $\alpha$ is continuous by Assumption~\ref{asu:alpha},
  this allows us to apply
  \cite[Theorem~1.5]{ethierMarkovProcessesCharacterization1986}, which finishes the proof.
\end{proof}
  \begin{theorem}[max-domain of attraction of the constructed max-id process]
    \label{thm:max_domain}
 Let $\alpha$ be a mass function satisfying Assumption~\ref{asu:alpha},
  let
$(L_t)_{t\ge 0}$ be a L\'evy process
satisfying Assumption~\ref{asu:levy}, and
    let $(X_t)_{t\ge 0}$ be the time-changed process in the sense of Lemma~\ref{lem:existence_tcp}.
    Consider the max-id process $(Z_t)_{t\ge 0}$ defined by
\begin{align*}
  Z_t
  \ = \
  \max
  \mathfrak{P}_t
  \left(
  (X_t)_{t\ge 0}, \alpha(x)\exp(-x)\mathrm{d}x
  \right)
  \,,
\end{align*}
let
\begin{align*}
    \zeta_n(t)
    \ := \
    \max_{k=1,\ldots,n} Z_t^{(k)} - \log(n)
    \,,
    \qquad n\in\NN
    \,,
  \end{align*}
  where $((Z_t^{(k)})_{t\ge 0})_{k=1,\ldots,n}$ are i.i.d.~copies of $(Z_t)_{t\ge 0}$, and let
\begin{align*}
    \zeta_L(t)
    \ := \
    \max \mathfrak{P}_t((L_t)_{t\ge 0},\exp(-x)\mathrm{d}x)
    \,.
  \end{align*}
  Then $\zeta_n$ converges weakly to $\zeta_L$ in f.d.d.
  If, furthermore, $(L_t)_{t\ge 0}$ has continuous paths,
  then convergence holds in $\mathbb{D}([0,\infty))$.
\end{theorem}
\begin{remark}
  The only L\'evy processes with continuous paths satisfying Assumption~\ref{asu:levy} are $(\sqrt{\sigma}(B_t-\sqrt{\sigma}t/2))_{t\ge 0}$, for $\sigma\ge 0$. For $\sigma=1$ we get the classical Brown--Resnick process $\zeta_{\mathrm{BR}}$ from the introduction.
  In view of \cite[Theorem~2.4]{linConvergenceExtremeValue2001},
  it is natural that convergence in $\mathbb{D}([0,\infty))$ can only hold when the limit is continuous.
  A concrete reason to support this is that for the extended modulus of continuity
  \begin{align*}
    \omega''(f,\delta, K)
    \ := \
    \sup_{
    \begin{smallmatrix}
      t_1, t_2\in K\\
      t_1 \le s \le t_2\,,
      |t_1-t_2|<\delta
    \end{smallmatrix}
    }
    |f(s)-f(t_1)|
    \land
    |f(s)-f(t_2)|
    \,,
  \end{align*}
  we cannot proceed as in
  \eqref{eq:tr_eq}.
\end{remark}
\begin{proof}[\textbf{Proof f.d.d}]
  Let $\ell\in\NN$, and $0\le t_1<\ldots,<t_\ell<\infty$.
  We want to show that
  \begin{align}
    \label{eq:goal:proof:fdd}
    \left(
    \zeta_n(t_{1})
    \,,\ldots,\,,
    \zeta_n(t_{\ell})
    \right)
    \ \to \
    \left(
    \zeta_L(t_{1})
    \,,\ldots,\,,
    \zeta_L(t_{\ell})
    \right)
    \qquad\text{in distribution as}\ n\to\infty\,.
  \end{align}
  To this end, let $x_1,\ldots,x_\ell\in\RR$ and let $\eta$ and $\eta_n$ be defined as in Lemma~\ref{lem:conv}.
  Note that since $\zeta_n$ is the maximum over $n$ independent and identically distributed max-id processes, it holds that
  \begin{align*}
    &
    \PP
    \left[
    \forall_{j\in\{1,\ldots,\ell\}}
    \
    \zeta_n(t_j)
    \
    \le
    \
    x_j
    \right]
    \\&
    \ = \
    \PP
    \left[
    \forall_{j\in\{1,\ldots,\ell\},k\in\{1,\ldots,n\}}
    \
    Z_{t_j}^{(k)}
    \
    \le
    \
    x_j +\log(n)
    \right]
    \\&
    \ = \
    \left(
    \PP
    \left[
    \forall_{j\in\{1,\ldots,\ell\}}
    \
    \max_{i\in\NN} X^{(i)}_{t_j}(U^{(i)})
    \
    \le
    \
    x_j +\log(n)
    \right]
    \right)^n
    \\&
    \ = \
    \left(
    \PP
    \left[
    \sum_{i\in\NN}
    \ind\{
    \exists_{j\in\{1,\ldots,\ell\}}
    \
    X^{(i)}_{t_j}(U^{(i)})
    > x_j +\log(n)
    \}
    \ = \ 0
    \right]
    \right)^n
    \\&
    \ = \
    \exp
    \left(
    -
    n
    \int_\RR
    \PP[
    \exists_{j\in\{1,\ldots,\ell\}}
    \
    X_{t_j}(x)
    > x_j +\log(n)
    ]
    \eta(x)
    \,\mathrm{d}x
    \right)
    \,,
      \end{align*}
      where the last step follows from Theorem~\ref{thm:brown}(i).
      Applying the inclusion-exclusion formula and Tonelli's theorem, we get
      \begin{align*}
        &
    \int_\RR
    \PP[
    \exists_{j\in\{1,\ldots,\ell\}}
    \
    X_{t_j}(x)
    > x_j +\log(n)
    ]
    \cdot
    \eta(x)
    \,\mathrm{d}x
        \\&
            \ = \
        \sum_{l=1}^\ell
        (-1)^{l+1}
        \sum_{\begin{smallmatrix}
          I\subseteq \{1,\ldots,\ell\}\\
          |I|=l
        \end{smallmatrix}}
            \int_\RR
    \PP[
        \forall_{j\in I}
    \
    X_{t_j}(x)
    > x_j +\log(n)
    ]
    \cdot
        \eta(x)
        \,\mathrm{d}x
        \,.
      \end{align*}
      By invariance of the process measure pair $((X_t)_{t\ge 0}, \eta)$
      (see Theorem~\ref{thm:pert_inv})
      we get for fixed $I\subseteq \{1,\ldots,\ell\}$ that
      \begin{align*}
        &
        n
            \int_\RR
    \PP[
        \forall_{j\in I}
    \
    X_{t_j}(x)
    > x_j +\log(n)
    ]
    \cdot
        \eta(x)
        \,\mathrm{d}x
        \\&
        \ = \
        n
            \int_{x_{\min\{j\in I\}}+\log(n)}^\infty
    \PP[
        \forall_{j\in I}
    \
        X_{t_j-t_{\min\{j\in I\}}}(x)
    > x_j +\log(n)
    ]
    \cdot
        \eta(x)
        \,\mathrm{d}x
        \\&
        \ = \
            \int_{x_{\min\{j\in I\}}}^\infty
    \PP[
        \forall_{j\in I}
    \
        X^n_{t_j-t_{\min\{j\in I\}}}(x)
    > x_j
    ]
    \cdot
        \eta_n(x)
        \,\mathrm{d}x
        \\&
        \ \to\
            \int_{x_{\min\{j\in I\}}}^\infty
    \PP[
        \forall_{j\in I}
    \
        (L_{t_j-t_{\min\{j\in I\}}}+x)
    > x_j
    ]
    \cdot
        \exp(-x)
        \,\mathrm{d}x
        \,.
      \end{align*}
      The convergence follows from
      \begin{align*}
    \PP[
        \forall_{j\in I}
    \
        X^n_{t_j-t_{\min\{j\in I\}}}(x)
    > x_j
    ]
    \cdot
        \eta_n(x)
        \ \le \
        \eta_n(x)
        \ \le \
        \overline{a}\exp(-x)
        \qquad\text{for all}
        \ x\in\RR
        \,,
      \end{align*}
      Lemma~\ref{lem:conv}
      and dominated convergence.
      Using the invariance
      of $((L_t)_{t\ge 0}, \exp(-x)\mathrm{d}x)$ we reverse the time shift,
and reversing the previous steps as well, we find that this expression equals
      \begin{align*}
        \PP
        \left[
        \forall_{j\in\{1,\ldots,\ell\}}
        \
        \zeta_L(t_j)\ \le \ x_j
        \right]
        \,.
      \end{align*}
      Since what we showed is equivalent to
      \eqref{eq:goal:proof:fdd}, this finishes the proof.
\end{proof}
\begin{proof}[\textbf{Proof of convergence in $\mathbb{D}([0,\infty))$}]
  By assumption, $(L_t)_{t\ge 0}\in \mathbb{C}([0,\infty))$. Theorem~\ref{thm:cad} then implies that $\zeta_L$, $ (Z_t)_{t\ge 0}$ and $\zeta_n$ have paths in $\mathbb{C}([0,\infty))$. Hence, it suffices to prove convergence in $\mathbb{C}([0,\infty))$.

  We prove convergence by establishing
  tightness. First note that $(\zeta_n(0))_{n\in\NN}$ converges weakly to a Gumbel distributed random variable and is therefore tight. By
  \cite[Theorem~7.3]{billingsleyConvergenceProbabilityMeasures1999}, showing tightness of $\zeta_n$ then amounts to controlling
  the modulus of continuity $\omega$.
  Let $K\subset [0,\infty)$ be an arbitrary compact set (possibly not containing $0$). For $f\in \mathbb{C}(K)$ and $\delta>0$ we define the modulus of continuity by
  \begin{align*}
\omega
  \left(
  f,\delta, K
  \right)
  &
  \ := \
  \sup_{
  \begin{smallmatrix}
    s,t\in K\\
    |s-t|<\delta
  \end{smallmatrix}
  }
    \left|
    f(t)-f(s)
        \right|
        \,.
  \end{align*}
  Our goal is to show
that for all $\varepsilon, a > 0$ there exists some $\delta>0$ and $N\in\NN$ such that
\begin{align}
  \label{eq:cond:tightness}
  \PP[\omega(\zeta_n, \delta, K) > a]
  < \varepsilon
  \qquad\text{for all}\ n \ > \ N
  \,.
\end{align}
  Together with the f.d.d. convergence, we get convergence in $\mathbb{C}(K)$ for an arbitrary compact set $K\subset  [0,\infty)$ and therefore convergence in $\mathbb{C}[0,\infty)$.
  We adapt parts of the proof of \cite[Theorem~17]{kabluchkoStationaryMaxstableFields2009} to our setting.
Write
  \begin{align*}
   Y_{k,n}(t)\ :=\ Z_{t}^{(k)}-\log(n)\,,\qquad t \ \ge \ 0\,.
  \end{align*}
  We structure the proof in several steps.
  \\
  \textbf{Decomposition}
  \\
  We decompose the event
  $
  \left\{
  \omega
  \left(
  \zeta_n, \delta, K
  \right)
  \ >\ a
  \right\}
  $.
Note that, by the properties of maxima and
 the invariance of the modulus of continuity $\omega$ under spatial shifts,
 we have
\begin{align}
  \label{eq:tr_eq}
  \begin{split}
  \omega
  \left(
  \zeta_n
  , \delta, K
  \right)
  &
  \ = \
  \sup_{
  \begin{smallmatrix}
    s,t\in K\\
    |s-t|<\delta
  \end{smallmatrix}
  }
    \left|
    \max_{k=1,\ldots,n}
    Z^{(k)}_t
    -
    \max_{k=1,\ldots,n}
    Z^{(k)}_s
    \right|
    \\&
    \ \le \
    \max_{k=1,\ldots,n}
  \sup_{
  \begin{smallmatrix}
    s,t\in K\\
    |s-t|<\delta
  \end{smallmatrix}
  }
    \left|
    Z^{(k)}_t
    -
    Z^{(k)}_s
    \right|
    \ = \
    \max_{k=1,\ldots,n}
    \omega(Y_{k,n},\delta, K)
    \,.
  \end{split}
\end{align}
Therefore
\begin{align}
  \label{eq:dec:c}
  \begin{split}
  &
  \left\{
  \omega
  \left(
  \zeta_n, \delta, K
  \right)
  \ >\ a
  \right\}
  \\
  &
  \ = \
  \left\{
  \omega
  \left(
  \zeta_n, \delta, K
  \right)
  \ >\ a
  \quad\text{and}\quad
    \exists_{t\in K}
    \
    \zeta_n(t)
    \neq
  \max_{
  \begin{smallmatrix}
  k=1,\ldots,n\\
    Y_{k,n}(0) > -c_2
  \end{smallmatrix}
  }
    Y_{k,n}(t)
  \right\}
  \\&
  \qquad\cup\
  \left\{
  \omega
  \left(
  \zeta_n, \delta, K
  \right)
  \ >\ a
  \quad\text{and}\quad
    \forall_{t\in K}
    \
    \zeta_n(t)
    =
  \max_{
  \begin{smallmatrix}
  k=1,\ldots,n\\
    Y_{k,n}(0) > -c_2
  \end{smallmatrix}
  }
    Y_{k,n}(t)
  \right\}
  \\&
  \subset\
\left\{
    \exists_{t\in K}
    \
    \zeta_n(t)
    \neq
  \max_{
  \begin{smallmatrix}
  k=1,\ldots,n\\
    Y_{k,n}(0) > -c_2
  \end{smallmatrix}
  }
    Y_{k,n}(t)
  \right\}
  \\&
  \qquad
  \ \cup \
  \left\{
  \omega
  \left(
  \zeta_n, \delta, K
  \right)
  \ >\ a
  \quad\text{and}\quad
    \forall_{t\in K}
    \
    \zeta_n(t)
     =
  \max_{
  \begin{smallmatrix}
  k=1,\ldots,n\\
    Y_{k,n}(0) > -c_2
  \end{smallmatrix}
  }
    Y_{k,n}(t)
  \quad\text{and}\quad
    \exists_{k\in\{1,\ldots,n\}}
    \
    Y_{k,n}(0) > -c_2
  \right\}
  \\&
  \ \subset  \
\left\{
    \exists_{t\in K}
    \
    \zeta_n(t)
    \neq
  \max_{
  \begin{smallmatrix}
  k=1,\ldots,n\\
    Y_{k,n}(0) > -c_2
  \end{smallmatrix}
  }
    Y_{k,n}(t)
  \right\}
  \ \cup \
  \ \bigcup_{k=1}^n\
  \left\{
  \omega
  \left(
  Y_{k,n}, \delta, K
  \right)
  \ >\ a
  \quad\text{and}\quad
    Y_{k,n}(0) > -c_2
  \right\}
  \\&
  \ =: \
  G_n \ \cup\  \bigcup_{k=1}^n C_{k,n}
  \,,
  \end{split}
\end{align}
where in the last step we used
  \eqref{eq:tr_eq}.
Next we define and estimate
\begin{align}
  \label{eq:dec:a}
  \begin{split}
  E_n
  &
  \ := \
  \left\{
  \inf_{t\in K} \zeta_n(t)
  \ <\ -c_1
  \right\}
  \\
  &
  \
  \subset
  \
  \bigcap_{k=1}^n
  \left\{
  \inf_{t\in K} Y_{k,n}(t)
  \ <\ -c_1
  \right\}
  \\&
  \
  \subset
  \
  \bigcap_{k=1}^n
  \left\{
  \inf_{t\in K} Y_{k,n}(t)
  \ <\ -c_1
  \qquad\text{or}\qquad
  Y_{k,n}(0)\ <\ -c_1
  \right\}
  \ =: \
  \bigcap_{k=1}^n
  A_{k,n}
  \,.
  \end{split}
\end{align}
With $E_n$ defined in this way, we have
 \begin{align}
   \label{eq:dec:b}
   \begin{split}
  G_n
  &
  \subset
  E_n \ \cup \
  \left\{
    \exists_{t^*\in K, k\in \{1,\ldots,n\}}
    \
    \zeta_n(t^*)
    \ = \
    Y_{k,n}(t^*)
    \,,
    \quad
    Y_{k,n}(0) \le -c_2
    \quad\text{and}\quad
    \inf_{t\in K}\zeta_n(t)
    \ \ge \ -c_1
  \right\}
  \\&
  \ \subset\
  E_n
  \ \cup \
  \bigcup_{k=1}^n
  \left\{
    Y_{k,n}(0) \ \le\  - c_2
        \qquad\text{and}\qquad
    \sup_{t\in K}Y_{k,n}(t)
    \ \ge \ -c_1
  \right\}
  \ =: \
   E_n \ \cup \ \bigcup_{k=1}^n B_{k,n}
  \,.
   \end{split}
\end{align}
Using \eqref{eq:dec:c}, \eqref{eq:dec:a}, and \eqref{eq:dec:b}, we get
\begin{align*}
  \left\{
  \omega
  \left(
  \zeta_n, \delta, K
  \right)
  \ >\ a
  \right\}
  \ \subset\
  \bigcap_{k=1}^n
  A_{k,n}
  \ \cup\
  \bigcup_{k=1}^n
  B_{k,n}
  \ \cup\
  \bigcup_{k=1}^n
  C_{k,n}
  \,.
\end{align*}
Since the events are i.i.d. for $k\in\{1,\ldots,n\}$, we have
\begin{align}
  \label{eq:bound:omega}
  \begin{split}
  \PP
  \left[
  \omega
  \left(
  \zeta_n, \delta, K
  \right)
  \ >\ a
  \right]
  &
  \ \le \
  (\PP[A_{1,n}])^n
  \ + \
  1
  \ - \
  \left(
    \PP[B_{1,n}^c]
  \right)^n
  \ + \
  n
  \PP[C_{1,n}]
  \,,
  \end{split}
\end{align}
where
\begin{align*}
  A_{1,n}
  &
  \ = \
  \left\{
  \inf_{t\in K} Y_{1,n}(t)
  \ < \ -c_1
  \qquad\text{or}\qquad
  Y_{1,n}(0)\ < \ -c_1
  \right\}
    \,,
  \\
  B_{1,n}^c
  &
  \ = \
  \left\{
    Y_{1,n}(0) \ >\  - c_2
        \qquad\text{or}\qquad
    \sup_{t\in K}Y_{1,n}(t)
    \ < \ -c_1
  \right\}
  \,,
  \\
  C_{1,n}
  &
  \ = \
  \left\{
  \omega
  \left(
  Y_{1,n}, \delta, K
  \right)
  \ >\ a
  \quad\text{and}\quad
    Y_{1,n}(0) > -c_2
  \right\}
  \,.
\end{align*}
\\
\textbf{Analysis of $(\PP[A_{1,n}])^n$}
\\
The goal is to show that there exists $c_1$ such that, for $N=N(c_1)$ sufficiently large, we have
\begin{align}
  \label{eq:a:0}
  \left(
  \PP[A_{1,n}]
  \right)^n
  \ \le \
  \frac{\varepsilon}{4}
  \qquad \text{for all} \ n\ge N(c_1)
  \,.
\end{align}
Observe that
\begin{align*}
  A_{1,n}
  &
  \ = \
  \left\{
  \inf_{t\in K}
  \max_{i\in\NN}
  X_t^{(i)}(U^{(i)})
  \ < \ -c_1 + \log(n)
  \qquad\text{or}\qquad
  \max_{i\in\NN} U^{(i)} \ < \ -c_1 + \log(n)
  \right\}
  \\&
  \ \subset\
  \left\{
  \max_{i\in\NN}
  \inf_{t\in K}
  X_t^{(i)}(U^{(i)})
  \ < \ -c_1 + \log(n)
  \qquad\text{or}\qquad
  \max_{i\in\NN} U^{(i)} \ < \ -c_1 + \log(n)
  \right\}
  \\&
  \ = \
  \left\{
  \forall_{i\in\NN}
  \,
  \inf_{t\in K}
  X_t^{(i)}(U^{(i)})
  \ < \ -c_1 + \log(n)
  \qquad\text{or}\qquad
  \forall_{i\in\NN}
  \,
   U^{(i)} \ < \ -c_1 + \log(n)
  \right\}
  \\&
  \ = \
  \bigcap_{i\in\NN}
  \left\{
  \inf_{t\in K}
  X_t^{(i)}(U^{(i)})
  \ < \ -c_1 + \log(n)
  \right\}
  \
  \cup
  \
  \bigcap_{i\in\NN}
  \left\{
   U^{(i)} \ < \ -c_1 + \log(n)
  \right\}
  \\&
  \ \subset\
  \bigcap_{i\in\NN}
  \left\{
  \inf_{t\in K}
  X_t^{(i)}(U^{(i)})
  \ < \ -c_1 + \log(n)
  \qquad\text{or}\qquad
  U^{(i)} \ < \ -c_1 + \log(n)
  \right\}
  \\&
  \ = \
  \bigcap_{i\in\NN}
  \left\{
  \inf_{t\in K}
  X_t^{(i)}(U^{(i)})
  \ \ge \ -c_1 + \log(n)
  \qquad\text{and}\qquad
  U^{(i)} \ \ge \ -c_1 + \log(n)
  \right\}^c
  \\&
  \ = \
  \left\{
  \sum_{i\in\NN}
  \ind\left\{
  \inf_{t\in K}
  X_t^{(i)}(U^{(i)})
  \ \ge \ -c_1 + \log(n)
  \qquad\text{and}\qquad
  U^{(i)} \ \ge \ -c_1 + \log(n)
  \right\}
  \ = \ 0
  \right\}
  \,,
\end{align*}
so that
\begin{align}
  \label{eq:a:1}
      \begin{split}
  &
  \left(
    \PP[A_{1,n}]
  \right)^n
    \\&
    \ \le \
    \left(
    \PP
    \left[
    \sum_{i\in \NN}
    \ind
    \left\{
    \inf_{t\in K} X^{(i)}_t(U^{(i)}) - \log(n)
  \ \ge \ -c_1
  \qquad\text{and} \qquad
    U^{(i)}\ \ge \ \log(n)-c_1
    \right\}
    \ = \ 0
      \right]
        \right)^{n}
      \,.
      \end{split}
\end{align}
To further bound this, note that, by Theorem~\ref{thm:brown} and the change of variables $x\mapsto x+\log(n)$,
\begin{align}
  \label{eq:a:2}
    \begin{split}
  &
  \left(
    \PP
    \left[
    \sum_{i\in \NN}
    \ind
    \left\{
    \inf_{t\in K} X^{(i)}_t(U^{(i)}) - \log(n)
  \ \ge \ -c_1
  \qquad\text{and} \qquad
    U^{(i)}\ \ge \ \log(n)-c_1
    \right\}
    \ = \ 0
      \right]
  \right)^n
      \\&
      \ = \
    \exp
    \left(
    -
    n
    \int_{\log(n)-c_1}^{\infty}
    \PP
    \left[
    \inf_{t\in K} X^{}_t(x) - \log(n)
  \ \ge \ -c_1
    \right]
    \eta(x)
    \,\mathrm{d}x
    \right)
    \\&
    \ = \
    \exp
    \left(
    -
    \int_{-c_1}^{\infty}
    \PP
    \left[
    \inf_{t\in K}
    X^{n}_t(x)
  \ \ge \ -c_1
    \right]
    \eta_n(x)
    \,\mathrm{d}x
    \right)
    \,,
    \end{split}
\end{align}
where
\begin{align*}
  X^{n}_t(x)
  \ = \
  X_t(x+\log(n))-\log(n)
  \,.
\end{align*}
To apply weak convergence of $(X_t^n)_{t\ge 0}$ to $(L_t)_{t\ge 0}$ (see Lemma~\ref{lem:conv}(ii)), we write
\begin{align}
  \label{eq:a:3}
    \begin{split}
  &
    \exp
    \left(
    -
    \int_{-c_1}^{\infty}
    \PP
    \left[
    \inf_{t\in K}
    X^{n}_t(x)
  \ \ge \ -c_1
    \right]
    \eta_n(x)
    \,\mathrm{d}x
    \right)
    \\&
    \ \le \
  \exp
    \left(
    -
    \int_{-c_1}^{\infty}
    \PP
    \left[
    \inf_{t\in K}
    L_t
    +
    x
  \ \ge \
  -c_1
    \right]
    \exp(-x)
    \,\mathrm{d}x
    \right)
    \\&
    \qquad + \
    \left|
    \exp
    \left(
    -
    \int_{-c_1}^{\infty}
    \PP
    \left[
    \inf_{t\in K} X^{n}_t(x)
  \ \ge \ -c_1
    \right]
    \eta_n(x)
    \,\mathrm{d}x
    \right)
    \right.
    \\&
    \left.
    \qquad\qquad - \
    \exp
    \left(
    -
    \int_{-c_1}^{\infty}
    \PP
    \left[
    \inf_{t\in K}
    L_t + x
  \ \ge \ -c_1
    \right]
    \exp(-x)
    \,\mathrm{d}x
    \right)
    \right|
    \,.
    \end{split}
\end{align}
For the first term of the upper bound in \eqref{eq:a:3},
making the change of variables $x\mapsto x+c_1$ we get
\begin{align*}
  &
    \int_{-c_1}^{\infty}
    \PP
    \left[
    \inf_{t\in K}
    L_t
    +
    x
  \ \ge \
  -c_1
    \right]
    \exp(-x)
    \,\mathrm{d}x
    \\
    &
    \ = \
    \exp(c_1)
    \int_{0}^{\infty}
    \PP
    \left[
    \inf_{t\in K}
    L_t
    +
    x
  \ \ge \
  0
    \right]
    \exp(-x)
    \,\mathrm{d}x
    \ >\ 0
    \,,
\end{align*}
so that, for $c_1$ large enough, it holds
\begin{align*}
  \exp
  \left(
  -
    \int_{-c_1}^{\infty}
    \PP
    \left[
    \inf_{t\in K}
    L_t
    +
    x
  \ \ge \
  -c_1
    \right]
    \exp(-x)
    \,\mathrm{d}x
  \right)
  \ \le \
  \frac{\varepsilon}{8}
  \,.
\end{align*}
For the second term of the upper bound in \eqref{eq:a:3}, note that analogously to the convergence proof for f.d.d. we get that
  \begin{align*}
    \int_{-c_1}^{\infty}
    \PP
    \left[
    \inf_{t\in K} X^{n}_t(x)
  \ \ge \ -c_1
    \right]
    \eta_n(x)
    \,\mathrm{d}x
    \ \to \
    \int_{-c_1}^{\infty}
    \PP
    \left[
    \inf_{t\in K}
    L_t + x
  \ \ge \ -c_1
    \right]
    \exp(-x)
    \,\mathrm{d}x
  \end{align*}
  for all $c_1\in \RR$ and as $n\to\infty$. We can therefore choose $n\ge N=N(c_1)$ large enough, such that
\begin{align*}
  &
    \left|
    \exp
    \left(
    -
    \int_{-c_1}^{\infty}
    \PP
    \left[
    \inf_{t\in K} X^{n}_t(x)
  \ \ge \ -c_1
    \right]
    \eta_n(x)
    \,\mathrm{d}x
    \right)
    \right.
    \\&
    \left.
    \qquad\qquad - \
    \exp
    \left(
    -
    \int_{-c_1}^{\infty}
    \PP
    \left[
    \inf_{t\in K}
    L_t + x
  \ \ge \ -c_1
    \right]
    \exp(-x)
    \,\mathrm{d}x
    \right)
    \right|
    \ \le \
    \frac{\varepsilon}{8}
    \,.
\end{align*}
Combining \eqref{eq:a:1}-\eqref{eq:a:3} yields \eqref{eq:a:0}.
  \\
  \textbf{Analysis of $1-\PP[B_{1,n}^c]^n$}
\\
The goal is to show that there exists $c_2=c_2(c_1)$ such that, for $N(c_1,c_2)$ sufficiently large, we have
\begin{align}
  \label{eq:b:0}
  1
  \ - \
  \left(
  \PP[B_{1,n}^{c}]
  \right)^n
  \ \le \
  \frac{\varepsilon}{4}
  \qquad \text{for all} \ n\ge N(c_1,c_2)
  \,.
\end{align}
  It holds
  \begin{align}
    \label{eq:b:1}
    \begin{split}
  &
    \PP[B_{1,n}^c]
    \ = \
    \PP
    \left[
    Y_{1,n}(0)\ > \ -c_2
    \qquad\text{or}\qquad
    \sup_{t\in K}
    Y_{1,n}(t)
    \ <  \
    -c_1
    \right]
    \\&
    \ = \
    1
    \ - \
    \PP
    \left[
    \max_{i\in\NN} U^{(i)}\ \le \ -c_2 + \log(n)
    \right]
    \\&
    \qquad +\
      \PP
    \left[
    \max_{i\in\NN}U^{(i)}\ \le \ -c_2 + \log(n)
    \qquad\text{and}\qquad
    \max_{i\in\NN}
    \sup_{t\in K}
    X_t^{(i)}(U^{(i)})
    \ <  \
    -c_1
    +\log(n)
    \right]
    \,,
    \end{split}
    \end{align}
    where the first probability in \eqref{eq:b:1} satisfies
    \begin{align}
      \label{eq:b:2}
       \begin{split}
    \PP
    \left[
    \max_{i\in\NN} U^{(i)}\ \le \ -c_2 + \log(n)
    \right]
      &
    \ = \
    \PP
    \left[
    \sum_{i\in\NN}
    \ind
    \left\{
      U^{(i)}\ > \ -c_2 + \log(n)
    \right\}
      \ = \ 0
    \right]
      \\&
      \ = \
   \exp
   \left(
   - \int_{-c_2+\log(n)}^{\infty}
   \alpha(x)\exp(-x)\,\mathrm{d}x
   \right)
       \,.
       \end{split}
  \end{align}
  For the second probability in \eqref{eq:b:1}, note that
    \begin{align}
      \label{eq:b:3}
  \begin{split}
      &
      \left\{
    \max_{i\in\NN} U^{(i)}\ \le \ -c_2 + \log(n)
    \qquad\text{and}\qquad
    \max_{i\in\NN}
    \sup_{t\in K}
    X_t^{(i)}(U^{(i)})
    \ <  \
    -c_1
    +\log(n)
      \right\}
      \\&
      \ = \
      \bigcap_{i\in\NN}
      \left\{
     U^{(i)}\ \le \ -c_2 + \log(n)
      \right\}
      \ \cap\
      \bigcap_{i\in\NN}
      \left\{
    \sup_{t\in K}
    X_t^{(i)}(U^{(i)})
    \ <  \
    -c_1
    +\log(n)
      \right\}
      \\&
      \ \subset \
  \bigcap_{i\in\NN}
      \left\{
     U^{(i)}\ \le \ -c_2 + \log(n)
      \right\}
      \\&
      \qquad
      \qquad
      \qquad
       \cap\
      \left\{
    \sup_{t\in K}
    X_t^{(i)}(U^{(i)})
    \ <  \
    -c_1
    +\log(n)
    \qquad\text{or}\qquad
    U^{(i)}> -c_2 + \log(n)
      \right\}
      \\&
      \ = \
           \bigcap_{i\in\NN}
      \left(
      \left\{
     U^{(i)}\ > \ -c_2 + \log(n)
      \right\}
      \right.
      \\&
      \left.
      \qquad
      \qquad
      \qquad
       \cup\
      \left\{
    \sup_{t\in K}
    X_t^{(i)}(U^{(i)})
    \ \ge  \
    -c_1
    +\log(n)
    \qquad\text{and}\qquad
    U^{(i)}\le -c_2 + \log(n)
      \right\}
      \right)^c
      \\
      &
    \ = \
    \left\{
    \sum_{i\in\NN}
    \ind\left\{U^{(i)}>-c_2+\log(n)
    \qquad\text{or}\qquad
    \right.
    \right.
    \\&
    \qquad
    \qquad
    \qquad
    \left.
    \left.
    \left(
    \sup_{t\in K}X^{(i)}(U^{(i)})\ge -c_1+\log(n)
    \qquad\text{and}\qquad
    U^{(i)}\le-c_2+\log(n)
    \right)
    \right\}
    = 0
    \right\}
    \,,
  \end{split}
    \end{align}
    where the probability of the latter set, by
  Theorem~\ref{thm:brown}, is
    \begin{align}
      \label{eq:b:4}
    \begin{split}
      &
    \exp
    \left(
    -
    \int_{-c_2+\log(n)}^{\infty}
    \alpha(x)\cdot\exp(-x)\ \mathrm{d}x
    \right.
    \\&
    \left.
    \qquad
    \qquad
    \qquad
     -\
    \int_{-\infty}^{-c_2+\log(n)}
    \PP[\sup_{t\in K}X_t(x)>-c_1+\log(n)]
    \cdot
    \alpha(x)\cdot\exp(-x)\ \mathrm{d}x
    \right)
    \,.
    \end{split}
  \end{align}
It follows from \eqref{eq:b:1}-\eqref{eq:b:4} that
  \begin{align}
    \label{eq:b:5}
    \begin{split}
  &
  1 - \PP[B_{1,n}^c]^n
  \\&
    \ = \
  1
  -
  \\&
  \qquad
  \left(
  1
  -
   \exp
   \left(
   - \int_{-c_2+\log(n)}^{\infty}
   \alpha(x)\exp(-x)\,\mathrm{d}x
   \right)
   \right.
   \\&
   \qquad
  \qquad
   +\
  \exp
    \left(
    -
    \int_{-c_2+\log(n)}^{\infty}
    \alpha(x)\cdot\exp(-x)\ \mathrm{d}x
    \right.
    \\&
    \left.
    \left.
    \qquad
    \qquad
    \qquad
  \qquad
    \qquad
     -\
    \int_{-\infty}^{-c_2+\log(n)}
    \PP[\sup_{t\in K}X_t(x)>-c_1+\log(n)]
    \cdot
    \alpha(x)\cdot\exp(-x)\ \mathrm{d}x
    \right)
    \right)^n
    \\&
    \ = \
     1
  -
  \\&
  \qquad
  \left(
  1
  -
   \exp
   \left(
   - \int_{-c_2+\log(n)}^{\infty}
   \alpha(x)\exp(-x)\,\mathrm{d}x
   \right)
   \right.
   \\&
   \qquad
  \qquad
  \left.
   \times\
   \left(
   1
    \ -\
    \exp\left(
     -\
    \int_{-\infty}^{-c_2+\log(n)}
    \PP[\sup_{t\in K}X_t(x)>-c_1+\log(n)]
    \cdot
    \alpha(x)\cdot\exp(-x)\ \mathrm{d}x
    \right)
    \right)
    \right)^n
    \\&
    \ = \
     1
  -
  \\&
  \qquad
  \left(
  1
  -
   \exp
   \left(
   -
   \frac{1}{n}
   \int_{-c_2}^{\infty}
   \alpha(x+\log(n))\exp(-x)\,\mathrm{d}x
   \right)
   \right.
   \\&
  \qquad
  \left.
   \times\
   \left(
   1
    \ -\
    \exp\left(
     -\
     \frac{1}{n}
    \int_{-\infty}^{-c_2}
    \PP[\sup_{t\in K}X_t(x+\log(n))>-c_1+\log(n)]
    \cdot
    \alpha(x+\log(n))\cdot\exp(-x)\ \mathrm{d}x
    \right)
    \right)
    \right)^n
    \\&
    \ \le \
      1
  -
  \\&
  \qquad
  \left(
  1
  -
   \exp
   \left(
   -
   \frac{1}{n}
   \int_{-c_2}^{\infty}
   \alpha(x+\log(n))\exp(-x)\,\mathrm{d}x
   \right)
   \right.
   \\&
  \qquad
  \left.
   \times\
     \frac{1}{n}
    \int_{-\infty}^{-c_2}
    \PP
    \left[
    \sup_{t\in K}X_t(x+\log(n))>-c_1+\log(n)
    \right]
    \cdot
    \alpha(x+\log(n))\cdot\exp(-x)\ \mathrm{d}x
    \right)^n
    \,,
    \end{split}
  \end{align}
  where in the last inequality we used
  $1-\exp(-x)\le x$.
  It remains to show
that, for $c_2=c_2(c_1)$ large enough, we have
  \begin{align}
    \label{eq:b:6}
    \begin{split}
  \int^{-c_2}_{-\infty}
  \PP
  \left[
  \sup_{t\in K}
    X_t(x+\log(n))
  \ge -c_1
  +\log(n)
  \right]
    \alpha(x+\log(n))
    \exp(-x)
  \,\mathrm{d}x
  \ \le \
  \frac{\varepsilon}{4}
  \,.
    \end{split}
  \end{align}
Then, for $n\ge N(c_1,c_2)$ sufficiently large,
  \begin{align*}
    &
    -
   \exp
   \left(
   -
   \frac{1}{n}
   \int_{-c_2}^{\infty}
   \alpha(x+\log(n))\exp(-x)\,\mathrm{d}x
   \right)
   \\&
  \qquad
   \times\
     \frac{1}{n}
    \int_{-\infty}^{-c_2}
    \PP[\sup_{t\in K}X_t(x+\log(n))>-c_1+\log(n)]
    \cdot
    \alpha(x+\log(n))\cdot\exp(-x)\ \mathrm{d}x
    \\&
    \ \ge \
    \frac{-\varepsilon}{4n}
    \ > \ -1
    \,,
  \end{align*}
  and applying
 Bernoulli's inequality together with
 \eqref{eq:b:5} gives
  \begin{align*}
    &
    1
    \ - \
    \left(
    \PP[B_{1,n}^c]
    \right)^n
    \\&
    \ \le \
    1
    \\&
    \qquad
    -\
  \left(
  1
  -
   \exp
   \left(
   -
   \frac{1}{n}
   \int_{-c_2}^{\infty}
   \alpha(x+\log(n))\exp(-x)\,\mathrm{d}x
   \right)
   \right.
   \\&
  \qquad
  \left.
   \times\
     \frac{1}{n}
    \int_{-\infty}^{-c_2}
    \PP[\sup_{t\in K}X_t(x+\log(n))>-c_1+\log(n)]
    \cdot
    \alpha(x+\log(n))\cdot\exp(-x)\ \mathrm{d}x
    \right)^n
    \\&
    \ \le \
    1
    \ - \
    \left(
    1
    -
    \frac{\varepsilon}{4}
    \right)
    \ = \
    \frac{\varepsilon}{4}
    \,.
      \end{align*}
      Let us show \eqref{eq:b:6}.
      Observe that
  \begin{align*}
    (-c_1+\log(n))-(x+\log(n))\ =\  - c_1 - x
    &\ >\  -c_1 + c_2
    \qquad\text{for}\ x \in (-\infty,-c_2)
    \,.
  \end{align*}
  Hence, there exists $c_2=c_2(c_1)$ large enough such that, by
  Lemma~\ref{lem:dynkin} and Assumption~\ref{asu:alpha},
  there exist $v>0$, depending only on $K$, and a constant $C_{v,K,\alpha}>0$ such that, for $x<-c_2$,
  \begin{align*}
    &
  \PP
  \left[
  \sup_{t\in K}
    X_t(x+\log(n))
    \ge -c_1+\log(n)
  \right]
    \\&
    \ \le \
C_{v,K,\alpha}
\cdot
    \PP[
    L_{\overline{K}\cdot \overline{\alpha}}
    \ge
    (-c_1+\log(n))
    -
    (x+\log(n))
    -
    v
    ]
    \\&
    \ = \
C_{v,K,\alpha}
\cdot
    \PP[
    L_{\overline{K}\cdot \overline{\alpha}}
    \ge
    -c_1 - x - v
        ]
        \,,
  \end{align*}
  and therefore
  \begin{align*}
    &
  \int^{-c_2}_{-\infty}
  \PP
  \left[
  \sup_{t\in K}
    X_t(x+\log(n))
    \ge -c_1+\log(n)
  \right]
    \alpha(x+\log(n))
    \exp(-x)
  \,\mathrm{d}x
    \\&
    \ \le \
C_{v,K,\alpha}
    \cdot
    \overline{\alpha}
    \int_{-\infty}^{-c_2}
    \PP
    \left[
    L_{\overline{K}\cdot\overline{\alpha}} \ge -c_1 - x - v
    \right]
    \exp(-x)
    \,\mathrm{d}x
    \\&
    \ = \
    C_{v,K,\alpha}
    \cdot
    \overline{\alpha}
    \cdot
    \exp(c_1+v)
    \int_{c_2-c_1- v}^{\infty}
    \PP
    \left[
    L_{\overline{K}\sup \alpha} \ge x
    \right]
    \exp(x)
    \,\mathrm{d}x
    \\&
    \ = \
C_{v,K,\alpha}
    \cdot
    \overline{\alpha}
    \cdot
    \exp(c_1+v)
    \cdot
    \EE
    \left[
    \exp(L_{\overline{K}\overline{\alpha}})
    \ind \left\{
    L_{\overline{K}\overline{\alpha}}
    \ge c_2 - c_1 - v
    \right\}
    \right]
    \,.
  \end{align*}
  Since $\EE[\exp(L_t)]=1$ for all $t\ge 0$, we can choose $c_2>c_1$ large enough to obtain \eqref{eq:b:6}.
  \\
  \textbf{Analysis of $n \PP[C_{1,n}]$}
\\
Following Remark~\ref{rem:order}, we
assume without loss of generality that $U^{(1)}=\max \{U^{(i)}\}_{i\in\NN}$.
We decompose
\begin{align*}
  &
  C_{1,n}
  \ = \
  \left\{
  \omega(Y_{1,n},\delta,K)
  > a \qquad\text{and}\qquad
  Y_{1,n}(0)
  > -c_2
  \right\}
  \\&
  \ \subset\
  \left\{
  \omega((X_t^{(1)}(U^{1})_{t\ge 0}), \delta, K)
  > a
  \quad\text{and}\quad
  U^{(1)}> - c_2 + \log(n)
  \quad\text{and}\quad
  (Z_t)_{t\in K}
  \ = \
  (X_t^{(1)}(U^{(1)}))_{t\in K}
  \right\}
  \\&
  \qquad \cup
  \left\{
  U^{(1)}
  > - c_2 + \log(n)
  \qquad\text{and}\qquad
  \exists_{j>1}\,
  \exists_{t\in K}\,
  X^{(j)}_t(U^{(j)})
   \ > \
  X^{(1)}_t(U^{(1)})
  \right\}
  \\&
  \
  \subset
  \
  D_{1,n}
  \
  \cup
  \
  D_{2,n}
  \,,
\end{align*}
where
\begin{align*}
  D_{1,n}
  &
  \ := \
  \left\{
  \omega((X_t^{(1)}(U^{(1)})_{t\ge 0}), \delta, K)
  > a
  \quad\text{and}\quad
  U^{(1)}> - c_2 + \log(n)
  \right\}
  \,,
  \\
  D_{2,n}
  &
  \ := \
  \left\{
  U^{(1)}
  > - c_2 + \log(n)
  \qquad\text{and}\qquad
  \exists_{j>1}\,
  \exists_{t\in K}\,
  X^{(j)}_t(U^{(j)})
   \ > \
  X^{(1)}_t(U^{(1)})
  \right\}
  \,,
\end{align*}
so that
\begin{align*}
  n\PP[C_{1,n}]
  \ \le \
  n
  \left(
  \PP[D_{1,n}]
  \ + \
  \PP[D_{2,n}]
  \right)
  \,.
\end{align*}
\textbf{Analysis of $n\PP[D_{1,n}]$}
\\
We show that there exists $\delta>0$ sufficiently small such that, for $n\ge N(c_1,c_2,\delta)$ sufficiently large,
\begin{align}
  \label{eq:d1:0}
  n
  \PP[D_{1,n}]
  \ \le \
  \frac{\varepsilon}{8}
  \,.
\end{align}
Observe that
\begin{align}
  \label{eq:dist:u1}
  \PP[U^{(1)}\le x]
  \ = \
  \exp
  \left(
  -\int_{x}^{\infty}
  \alpha(s)
  \exp(-s)
  \,\mathrm{d}s
  \right)
  \,,
  \qquad
  x\in\RR
  \,,
\end{align}
so that
\begin{align}
  \label{eq:d1:1}
  \begin{split}
  n\cdot\PP[D_{1,n}]
  &
  \ = \
  n
  \cdot
  \int_{-c_2+\log(n)}^{\infty}
  \PP[\omega((X_t(x))_{t\ge 0},\delta, K) > a]
  \,\mathrm{d}\PP_{U^{(1)}}(x)
  \\&
  \ = \
  n
  \cdot
  \int_{-c_2+\log(n)}^{\infty}
  \PP[\omega((X_t(x))_{t\ge 0},\delta, K) > a]
  \cdot
  \alpha(x)
  \cdot
  \exp(-x)
  \cdot
  \PP[U^{(1)}\le x]
  \,\mathrm{d}x
  \\&
  \ \le \
  \overline{\alpha}
  \int_{-c_2}^{\infty}
  \PP[\omega((X_t(x+\log(n))-\log(n))_{t\ge 0},\delta, K) > a]
  \cdot
  \exp(-x)
  \,\mathrm{d}x
  \\&
  \ = \
  \overline{\alpha}
  \int_{-c_2}^{\infty}
  \PP[\omega((X_t^n)_{t\ge 0},\delta, K) > a]
  \cdot
  \exp(-x)
  \,\mathrm{d}x
  \,,
  \end{split}
\end{align}
where in the last inequality we used a change of variables and that $\omega$ is invariant under spatial shift.
To apply weak convergence of $(X_t^n)_{t\ge 0}$ to $(L_t)_{t\ge 0}$, we write
\begin{align}
  \label{eq:d1:2}
  \begin{split}
  &
  \int_{-c_2}^{\infty}
  \PP[\omega((X_t(x+\log(n))-\log(n))_{t\ge 0},\delta, K) > a]
  \cdot
  \exp(-x)
  \,\mathrm{d}x
  \\&
  \ \le \
  \int_{-c_2}^{\infty}
  \PP[\omega((L_t)_{t\ge 0},\delta, K) > a]
  \cdot
  \exp(-x)
  \,\mathrm{d}x
  \\&
  \qquad + \
  \left|
  \int_{-c_2}^{\infty}
  \PP[\omega((X_t(x+\log(n))-\log(n))_{t\ge 0},\delta, K) > a]
  \cdot
  \exp(-x)
  \,\mathrm{d}x
  \right.
  \\
  &
  \qquad
  \qquad
  \left.
   - \
  \int_{-c_2}^{\infty}
  \PP[\omega((L_t)_{t\ge 0},\delta, K) > a]
  \cdot
  \exp(-x)
  \,\mathrm{d}x
  \right|
  \,.
  \end{split}
\end{align}
For the first term in the upper bound of \eqref{eq:d1:2}, note that
since $(L_{t})_{t\ge 0}$ has continuous paths we can choose $\delta(c_2)>0$ small enough, such that
\begin{align*}
  \PP[\omega((L_t)_{t\ge 0},\delta, K) > a]
  \ \le \
  \frac{\varepsilon}{16}
  \exp(-c_2)
  \,,
\end{align*}
so that
\begin{align}
  \label{eq:d1:3}
  \int_{-c_2}^{\infty}
  \PP[\omega((L_t)_{t\ge 0},\delta, K) > a]
  \cdot
  \exp(-x)
  \,\mathrm{d}x
  \ \le \
  \frac{\varepsilon}{16}
  \,.
\end{align}
For the second term in the upper bound of \eqref{eq:d1:2} note that
since
\begin{align*}
  \int_{-c_2}^{\infty}
  \exp(-x)
  \,\mathrm{d}x
  \ = \
  \exp(c_2)
  \ < \ \infty
  \,,
\end{align*}
and, by Lemma~\ref{lem:conv}(ii), as $n\to\infty$,
\begin{align*}
  \PP[\omega((X_t(x+\log(n))-\log(n))_{t\ge 0},\delta, K) > a]
  \ \to \
  \PP[\omega((L_t)_{t\ge 0},\delta, K) > a]
  \,,
\end{align*}
so that we can choose $N(c_1,c_2,\delta)$ large enough, such that for $n\ge N(c_1,c_2,\delta)$ it holds
\begin{align}
  \label{eq:d1:4}
  \begin{split}
  &
  \left|
  \int_{-c_2}^{\infty}
  \PP[\omega((X_t(x+\log(n))-\log(n))_{t\ge 0},\delta, K) > a]
  \cdot
  \exp(-x)
  \,\mathrm{d}x
  \right.
  \\
  &
  \qquad
  \qquad
  \left.
   - \
  \int_{-c_2}^{\infty}
  \PP[\omega((L_t)_{t\ge 0},\delta, K) > a]
  \cdot
  \exp(-x)
  \,\mathrm{d}x
  \right|
  \\&
  \ \le \
  \frac{\varepsilon}{16}
  \,.
  \end{split}
\end{align}
Putting together \eqref{eq:d1:1}-\eqref{eq:d1:4} yields \eqref{eq:d1:0}.
\\
\textbf{Analysis of $n\PP[D_{2,n}]$}
\\
We show that there exists $N\in \NN$ sufficiently larger such that
\begin{align}
  \label{eq:d2:0}
  n\PP[D_{2,n}]
  \ \le \
  \frac{\varepsilon}{8}
  \qquad\text{for all}\ n\ge N
  \,.
\end{align}
\textbf{Step 1}
\\
By \eqref{eq:d1:1} it holds
\begin{align}
  \label{eq:d2:1}
  \begin{split}
  &
  n\PP[D_{2,n}]
  \ = \
  n
  \PP
  \left[
  U^{(1)}>-c_2 + \log(n)
  \quad\text{and}\quad
  \exists_{j>1}\,\exists_{t\in K}\,
  X_t^{(j)}(U^{(j)})
  \
  \ge
  \
  X_t^{(1)}(U^{(1)})
  \right]
  \\&
  \ = \
  n
  \int_{-c_2+\log(n)}^{\infty}
  \PP
  \left[
  \exists_{j>1}\,\exists_{t\in K}\,
  X_t^{(j)}(U^{(j)})
  \
  \ge
  \
  X_t^{(1)}(U^{(1)})
  \mid U^{(1)}=x
  \right]
  \,\mathrm{d}\PP_{U^{(1)}}(x)
  \\&
  \ = \
  n
  \int_{-c_2+\log(n)}^{\infty}
  \PP
  \left[
  \exists_{j>1}\,\exists_{t\in K}\,
  X_t^{(j)}(U^{(j)})
  \
  \ge
  \
  X_t^{(1)}(U^{(1)})
  \mid U^{(1)}=x
  \right]
  \\&
  \qquad
  \qquad
  \times
  \alpha(x)
  \cdot
  \exp(-x)
  \cdot
  \PP[U^{(1)}\le x]
  \,\mathrm{d}x
  \\&
  \ = \
  \int_{-c_2}^{\infty}
  \PP
  \left[
  \exists_{j>1}\,\exists_{t\in K}\,
  X_t^{(j)}(U^{(j)})
  \
  \ge
  \
  X_t^{(1)}(U^{(1)})
  \mid U^{(1)}=x+\log(n)
  \right]
  \\&
  \qquad
  \qquad
  \times
  \alpha(x+\log(n))
  \cdot
  \exp(-x)
  \cdot
  \PP[U^{(1)}\le x+\log(n)]
  \,\mathrm{d}x
  \\&
  \ \le \
  \overline{\alpha}
  \exp(c_2)
  \sup_{x>-c_2}
  \PP
  \left[
  \exists_{j>1}\,\exists_{t\in K}\,
  X_t^{(j)}(U^{(j)})
  \
  \ge
  \
  X_t^{(1)}(U^{(1)})
  \mid U^{(1)}=x+\log(n)
  \right]
  \,.
  \end{split}
\end{align}
\textbf{Step 2}
\\
We define, for $g\in \mathbb{C}([0,\infty))$, the transformation $\widetilde{g}_{}$ by
\begin{align*}
  \widetilde{g}_{}
  \
  \colon
  \
  [0,\infty)\times \RR
  \ \to \
  [0,\infty)
  \,,\qquad
  (t,z)
  \ \mapsto\
  \left(
  \int_0^{(\cdot)}
  \alpha(g(r)
  + z)
  \,\mathrm{d}r
  \right)^{-1}
  (t)
  \,.
\end{align*}
Then, conditional on $(L_t^{(1)})_{t\ge 0}=g$, for some $g\in \mathbb{C}([0,\infty))$, we get
\begin{align*}
  X_t^{(1)}
  (z)
  \ = \
  g(\widetilde{g}_{}(t,z))
  \ + \
  z
    \,.
\end{align*}
By
independence of $(L^{(1)}_t)_{t\ge 0}$ from the remaining variables,
\begin{align}
  \label{eq:d2:2}
  \begin{split}
  &
  \PP
  \left[
  \exists_{j>1}\,\exists_{t\in K}\,
  X_t^{(j)}(U^{(j)})
  \
  \ge
  \
  X_t^{(1)}(U^{(1)})
  \mid U^{(1)}=x+\log(n)
  \right]
  \\&
  \ \le \
  \PP
  \left[
  \exists_{j>1}\,
  \sup_{t\in K}
  X_t^{(j)}(U^{(j)})
  \
  \ge
  \
  \inf_{t\in K}
  X_t^{(1)}(x+\log(n))
  \mid U^{(1)}=x+\log(n)
  \right]
  \\&
  \ = \
  \PP
  \left[
  \sum_{j>1}
  \ind\{
  \sup_{t\in K}
  X_t^{(j)}(U^{(j)})
  \
  \ge
  \
  \inf_{t\in K}
  X_t^{(1)}(x+\log(n))
  \}
  \ > \ 0
  \
  \Bigg|
  \
  U^{(1)}=x+\log(n)
  \right]
  \\&
  \ = \
  \int_{\mathbb{C}([0,\infty))}
  \,\mathrm{d}\PP_{(L^{(1)}_t)_{t\ge 0}}(g)
  \\&
  \qquad
  \PP
  \left[
  \sum_{j>1}
  \ind\{
  \sup_{t\in K}
  X_t^{(j)}(U^{(j)})
  \
  \ge
  \
  \inf_{t\in K}
  g(\widetilde{g}(t, x+\log(n)))
  \ + \
  x+\log(n)
  \}
  \ > \ 0
  \
  \Bigg|
  \
  U^{(1)}=x+\log(n)
  \right]
  \\&
  \ \le \
  \int_{\mathbb{C}([0,\infty))}
  \ind\left\{
  \inf_{t\in [0,\overline{K}/\underline{\alpha}]}g(t)>-R
  \right\}
  \,\mathrm{d}\PP_{(L^{(1)}_t)_{t\ge 0}}(g)
  \\&
  \qquad
  \PP
  \left[
  \sum_{j>1}
  \ind\{
  \sup_{t\in K}
  X_t^{(j)}(U^{(j)})
  \
  \ge
  \
  \inf_{t\in K}
  g(\widetilde{g}(t, x+\log(n)))
  \ + \
  x+\log(n)
  \}
  \ >\ 0
  \
  \Bigg|
  \
  U^{(1)}=x+\log(n)
  \right]
  \\&
  \qquad
  + \
  \PP
  \left[
  \inf_{t\in [0,\overline{K}/\underline{\alpha}]}L_t\le -R
  \right]
  \,,
  \end{split}
\end{align}
where
\begin{align}
  \label{eq:d2:3}
  \PP
  \left[
  \inf_{t\in [0,\overline{K}/\underline{\alpha}]}L_t\le -R
  \right]
  \ < \
  \frac{\varepsilon}{16}
  \exp(-c_2)
  \,,
\end{align}
for $R>0$ large enough.
\\
\textbf{Step 3}
\\
Write
\begin{align*}
 S
  \ := \
  \left\{
  \widetilde{g}_{}(t,z)
  \mid
  t\in K
  \,,
  z\in\RR
  \right\}
  \ = \
  \bigcup_{z\in\RR}
  \left\{
  \widetilde{g}_{}(t,z)
  \mid
  t\in K
  \right\}
  \ =: \
  \bigcup_{z\in\RR}
  S_z
  \,,
\end{align*}
where
\begin{align*}
  S_z
  \ = \
  \left(
  \widetilde{g}_{}(\cdot,z)
  \right)^{-1}
  (K)
  \
  \subset
  \
  \left(
  \widetilde{g}_{}(\cdot,z)
  \right)^{-1}
  ([0,\overline{K}])
  \
  \subset
  \
  [0,\overline{K}/\underline{\alpha}]
  \,,
\end{align*}
and therefore $S\subset
  [0,\overline{K}/\underline{\alpha}]$, independent of the choice of $g\in \mathbb{C}([0,\infty))$.
  Hence,
\begin{align*}
  \inf_{z\in\RR}
  \inf_{t\in K}
  g(\widetilde{g}(t, z))
  \ = \
  \inf_{t\in S}
  g(t)
  &
  \ \ge \
  \inf_{t\in [0,\overline{K}/\underline{\alpha}]}
  g(t)
   \,,
\end{align*}
so that
\begin{align}
  \label{eq:d2:4}
  \begin{split}
  &
 \int_{\mathbb{C}([0,\infty))}
  \ind\left\{
    \inf_{t\in [0,\overline{K}/\underline{\alpha}]}g(t)>-R
  \right\}
  \,\mathrm{d}\PP_{(L^{(1)}_t)_{t\ge 0}}(g)
  \\&
  \qquad
  \PP
  \left[
  \sum_{j>1}
  \ind\{
  \sup_{t\in K}
  X_t^{(j)}(U^{(j)})
  \
  \ge
  \
  \inf_{t\in K}
  g(\widetilde{g}(t, x+\log(n)))
  \ + \
  x+\log(n)
  \}
  \ >\ 0
  \
  \Bigg|
  \
  U^{(1)}=x+\log(n)
  \right]
  \\&
  \ \le \
  \PP
  \left[
  \sum_{j>1}
  \ind\{
  \sup_{t\in K}
  X_t^{(j)}(U^{(j)})
  \
  \ge
  \
  -R
  \ + \
  x+\log(n)
  \}
  \ >\ 0
  \
  \Bigg|
  \
  U^{(1)}=x+\log(n)
  \right]
  \,.
  \end{split}
\end{align}
\textbf{Step 4}
\\
By the Slivniak-Mecke Theorem (see for example \cite[Proposition~13.1.VII]{daleyIntroductionTheoryPoint2008}),
$\{U^{(j)}\}_{j>1}$ conditioned on $U^{(1)}=x+\log(n)$ is a PPP with intensity measure $\eta[\cdot \cap (-\infty,x+\log(n))]$, so that $\{\sup_{t\in K}X^{(j)}(U^{(j)})\}_{j>1}$ is a PPP on the real line with intensity
\begin{align*}
  \int_{-\infty}^{x+\log(n)}
  \PP
  \left[
  \sup_{t\in K}X_t(z)\in A
  \right]
  \eta(z)
  \,\mathrm{d}z
  \qquad\text{for all}\ A \in \mathcal{B}(\RR)
  \,.
\end{align*}
It therefore holds
\begin{align}
  \label{eq:d2:5}
  \begin{split}
  &
  \PP
  \left[
  \sum_{j>1}
  \ind\{
  \sup_{t\in K}
  X_t^{(j)}(U^{(j)})
  \
  \ge
  \
  -R
  \ + \
  x+\log(n)
  \}
  \ > \ 0
  \
  \Bigg|
  \
  U^{(1)}=x+\log(n)
  \right]
  \\&
  \ = \
  1
  \\&
  \qquad
  - \
  \PP
  \left[
  \sum_{j>1}
  \ind\{
  \sup_{t\in K}
  X_t^{(j)}(U^{(j)})
  \
  \ge
  \
  -R
  \ + \
  x+\log(n)
  \}
  \ = \ 0
  \
  \Bigg|
  \
  U^{(1)}=x+\log(n)
  \right]
  \\&
  \ = \
  1 \ - \
  \exp
  \left(
  -
  \int_{-\infty}^{x+\log(n)}
  \PP
  \left[
  \sup_{t\in K}X_t(z)
\
  \ge
  \
  -R
  \ + \
  x+\log(n)
  \right]
  \eta(z)
  \,\mathrm{d}z
  \right)
  \\&
  \ \le \
  \int_{-\infty}^{x+\log(n)}
  \PP
  \left[
  \sup_{t\in K}X_t(z)
\
  \ge
  \
  -R
  \ + \
  x+\log(n)
  \right]
  \eta(z)
  \,\mathrm{d}z
  \,.
  \end{split}
\end{align}
\\
\textbf{Step 5}
\\
Observe that
\begin{align}
  \label{eq:d2:6}
  \begin{split}
  &
  \int_{-\infty}^{x+\log(n)}
  \PP
  \left[
  \sup_{t\in K}X_t(z)
\
  \ge
  \
  -R
  \ + \
  x+\log(n)
  \right]
  \eta(z)
  \,\mathrm{d}z
  \\&
  \ \le \
  \int_{x+\log(n)-R-v}^{x+\log(n)}
  \eta(z)
  \,\mathrm{d}z
  \\&
  \qquad
   + \
  \int_{-\infty}^{x+\log(n)-R-v}
  \PP
  \left[
  \sup_{t\in K}X_t(z)
\
  \ge
  \
  -R
  \ + \
  x+\log(n)
  \right]
  \eta(z)
  \,\mathrm{d}z
  \,,
  \end{split}
  \end{align}
  where we choose $v$ as in Lemma~\ref{lem:dynkin}, and note that for $x>-c_2$
  \begin{align}
    \label{eq:d2:7}
    \begin{split}
    &
  \int_{x+\log(n)-R-v}^{x+\log(n)}
  \eta(z)
  \,\mathrm{d}z
  \ \le \
    \overline{\alpha}
    \frac{1}{n}
    \exp(-x)
    \left(
     \exp(R + v) - 1
    \right)
    \\&
    \ \le \
    \exp(c_2)
    \overline{\alpha}
    \frac{1}{n}
    \left(
    \exp(R + v) - 1
    \right)
    \\&
    \ \le \
    \exp(-c_2)
    \frac{\varepsilon}{16}
    \,,
    \end{split}
  \end{align}
  for $n\ge N(c_1,c_2,\delta,v,R)$ large enough.
  \\
  \textbf{Step 6}
  \\
  By Lemma~\ref{lem:dynkin}, there exists a constant $C_{v,\overline{K}}>0$ such that
  \begin{align}
    \label{eq:d2:8}
    \begin{split}
    &
  \int_{-\infty}^{x+\log(n)-R-v}
  \PP
  \left[
  \sup_{t\in K}X_t(z)
\
  \ge
  \
  -R
  \ + \
  x+\log(n)
  \right]
  \eta(z)
  \,\mathrm{d}z
    \\&
    \ \le \
    \overline{\alpha}
  \int_{-\infty}^{x+\log(n)-R-v}
  \PP
  \left[
    L_{\overline{\alpha}\cdot \overline{K}}
  \
  \ge
  \
  -R
  \ + \
  x+\log(n)
  - z - v
  \right]
    \exp(-z)
  \,\mathrm{d}z
    \\&
    \ = \
    \overline{\alpha}
    \exp(R+v)
    \exp(-x)
    \frac{1}{n}
    \int_0^{\infty}
  \PP
  \left[
    L_{\overline{\alpha}\cdot \overline{K}}
  \
  \ge
  \
  z
  \right]
    \exp(-z)
  \,\mathrm{d}z
    \\&
    \ \le \
    \overline{\alpha}
    \exp(R+v)
    \exp(c_2)
    \frac{1}{n}
    \EE[\exp(L_{\overline{\alpha}\cdot \overline{K}})]
    \\&
    \ \le \
    \exp(-c_2)
    \frac{\varepsilon}{16}
    \,,
    \end{split}
  \end{align}
  for $n\ge N(c_1,c_2,\delta,v,R)$ large enough.
  Combining
  \eqref{eq:d2:1}-\eqref{eq:d2:8} yields
\begin{align*}
  n\PP[D_{2,n}] \ \le \
  \varepsilon
  \exp(c_2)
  \exp(-c_2)
  \left(
  \frac{1}{16}
  \ + \
  \frac{1}{16}
  \right)
  \ = \
  \frac{\varepsilon}{8}
  \,,
\end{align*}
that is, \eqref{eq:d2:0}.
This finishes the proof.
\end{proof}
\section*{Acknowledgments}
The author is grateful to Marco Oesting for many fruitful discussions. The author also wants to thank Lukas Trottner for introducing him to \cite{azemaMesureInvarianteProcessus1969}.

\end{document}